\documentclass[twoside,final,11pt,a4paper]{amsart}

\usepackage[utf8]{inputenc}
\usepackage[T1]{fontenc}
\usepackage[english]{babel}
\usepackage{dsfont}
\usepackage{textcomp}
\usepackage{graphicx}
\usepackage{enumerate}
\usepackage[dvipsnames,svgnames]{xcolor}
\usepackage{bbm}
\usepackage{a4wide}
\usepackage{pdfpages}
\usepackage{epstopdf}

 \usepackage{hyperref}
 \hypersetup{colorlinks=true,linkcolor=blue,citecolor=blue,linktoc=page}
 
\usepackage[left]{showlabels}
 
\newcommand{\Exp}{\mathds{E}}   
\newcommand{\Prob}{\mathds{P}}  
 
\def\P{\mathds{P}}

\newcommand{\leqslant}{\leq}
\newcommand{\geqslant}{\geq}

\def\bbR{\mathbb{R}}

\newcommand{\M}{\mathbf{M}}

\newcommand{\N}{\mathds{N}}                
\newcommand{\R}{\mathds{R}}                     

\renewcommand{\H}{\mathbf H_{e}}

\newcommand{\cmax}{\mathfrak{c}_{\mathrm{max}}}     
\newcommand{\cmin}{\mathfrak{c}_{\mathrm{min}}}

\newcommand{\B}{\mathbf{C}}
\newcommand{\C}{\mathbf{C}}
\newcommand{\X}{\mathbf{X}}
\newcommand{\tB}{\widetilde\B}
\newcommand{\Tr}[1]{\mathfrak{tr}[{#1}]}

\definecolor{freeblue}{rgb}{0.25,0.41,0.88}
\definecolor{redish}{rgb}{1.0, 0.54117647058, 0.0}
\definecolor{best}{rgb}{0.0, 0.4470588235294118, 0.6980392156862745}
\definecolor{annuluscolor}{rgb}{0.33725490196078434, 0.7058823529411765, 0.9137254901960784}

\usepackage[textwidth=2.8cm, textsize=scriptsize]{todonotes} 
\newcommand{\eq}{\begin{equation}}
\newcommand{\qe}{\end{equation}}

\theoremstyle{plain}
\newtheorem{thm}{Theorem}
\newtheorem{lem}[thm]{Lemma}
\newtheorem{prop}[thm]{Proposition}

\newtheorem*{rem}{Remark}

\theoremstyle{definition}
\theoremstyle{remark}

\numberwithin{equation}{section}

\def\brho{\bar\rho}

\title{Sparse Recovery from Extreme Eigenvalues Deviation Inequalities}
\date{\today}

\author{Sandrine Dallaporta}
\address{SD is with CMLA, ENS Cachan, CNRS, Universit\'e Paris-Saclay, 94235 Cachan, France.}
\email{sandrine.dallaporta@cmla.ens-cachan.fr}
\author{Yohann De Castro}
\address{YDC is with the Laboratoire de Math\'ematiques d'Orsay, Univ. Paris-Sud, CNRS, Universit\'e Paris-Saclay, 91405 Orsay, France.}
\email{yohann.decastro@math.u-psud.fr}
\address{YDC is with the MOKAPLAN team at INRIA Paris, 2 rue Simone Iff, 75012 Paris, France.}
\email{yohann.de-castro@inria.fr}
\address{YDC is with the CERMICS Laboratory at Ponts ParisTech, 6 et 8 avenue Blaise Pascal, 77455 Marne la Vallée Cedex 2, France.}
\email{yohann.de-castro@enpc.fr}

\begin{document}

\begin{abstract}
This article provides a new toolbox to derive sparse recovery guarantees\textemdash that is referred to as ‘‘stable and robust sparse regression''~(SRSR)\textemdash from deviations on extreme singular values or extreme eigenvalues obtained in Random Matrix Theory. This work is based on Restricted Isometry Constants (RICs) which are a pivotal notion in Compressed Sensing and High-Dimensional Statistics as these constants finely assess how a linear operator is conditioned on the set of sparse vectors and hence how it performs in SRSR. While it is an open problem to construct deterministic matrices with apposite RICs, one can prove that such matrices exist using random matrices models. In this paper, we show upper bounds on RICs for Gaussian and Rademacher matrices using state-of-the-art deviation estimates on their extreme eigenvalues. This allows us to derive a lower bound on the probability of getting SRSR. One benefit of this paper is a direct and explicit derivation of upper bounds on RICs and lower bounds 
on SRSR from deviations on the extreme eigenvalues given by Random Matrix theory.
\end{abstract}
\maketitle

\section{Introduction}
\subsection{Stable and Robust Sparse Recovery (SRSR)}
The recent breakthrough of Compressed Sensing \cite{MR2412803,MR2243152,MR2236170} has shown that one can acquire and compress  simultaneously a signal from few linear measurements. This methodology has an important impact in practice since it may be deployed in applied contexts where the time of acquisition is limited\textemdash {\it e.g.} medical imaging (MRI and functional~MRI)\textemdash  and/or costly\textemdash {\it e.g.} reflection seismology, one may consult \cite{CGLP,foucart2013mathematical} and references therein.

More precisely, the problem addressed in recent researches aims at solving under-determined systems of linear equations (with an additive error term $\mathbf{e}$) of the form
\begin{equation}
\label{eq:LinModelNoise}
y=\M x_0+\mathbf{e}
\end{equation}
where $\M$ is a known $(n \times p)$ matrix, $x_0$ an unknown vector in $\R^p$, $y$ and $\mathbf{e}$ are vectors in $\R^n$ and~$n$ is (much) smaller than $p$. This frame fits many interests across various fields of research, {\it e.g.} in statistics one would estimate $p$ parameters $x_0$ from a sample $y$ of size~$n$, $\M$ being the design matrix and $\mathbf{e}$ some random centered noise. Although the matrix $\M$ is not injective, recent advances have shown that one can recover an interesting estimate $\hat x$ of $x_0$ considering $\ell_1$-minimization solutions  as 
\eq
\label{eq:L1minNoise}
\hat x\in\arg\min\|x\|_1\quad\mathrm{s.t.}\quad\|y-\M x\|_2\leq\eta\,,
\qe
where $\eta>0$ is a tuning parameter such that the experimenter believes it holds $\|\mathbf{e}\|_2\leq\eta$ with high probability. 


A standard goal is to prove that the estimate $\hat x$ is close to $x_0$ or more precisely that the norm of the error $\hat x-x_0$ is small. To this purpose, one says that $\hat x$ satisfies the Stable and Robust Sparse Recovery (SRSR, see~\cite[page 88]{foucart2013mathematical}) if the following two error bounds hold
\begin{align}
\label{eq:SRSR1}
\tag{$\ell_1\textnormal{-}\mathbf{SRSR}$}
\|x_0-\hat x\|_1&\leq C\sigma_s(x_0)_1+D\sqrt s\eta\\
\label{eq:SRSR2}
\tag{$\ell_2\textnormal{-}\mathbf{SRSR}$}
\|x_0-\hat x\|_2&\leq \frac{C}{\sqrt s}\sigma_s(x_0)_1+D\eta
\end{align}
where $C,D>0$ are constants and $\sigma_s(x_0)_1$ denotes the approximation error in $\ell_1$-norm by~$s$ coefficients, namely 
\[
\sigma_s(x_0)_1:=\min\|x_0-x\|_1\,,
\]
where the minimum is taken over the space~$\Sigma_{s}$ of sparse vectors $x$, {\it i.e.} the set of vectors with at most $s$ nonzero coordinates. 
SRSR shows that the estimate $\hat x$ recovers the $s$ largest coefficients of the target vector $x_0$ in a stable\footnote{In an idealized situation one would assume that $x_0$ is sparse. Nevertheless, in practice, we can only claim that $x_0$ is close to sparse vectors. The stability is the ability to control the estimation error $\| x_0-\hat x\|$ by the distance between $x_0$ and the sparse vectors. The reader may consult \cite[page 82]{foucart2013mathematical} for instance.}
and robust (to additive errors $\mathbf e$) manner. One can prove that the~$\sqrt s$ terms are optimal in the sense of Approximation Theory, one may consult \cite[Remark~4.2.3 page 88]{foucart2013mathematical} for further details. Sufficient condition for SRSR holds whenever the matrix $X$ satisfies some properties, see for instance \cite{MR2236170,MR2300700,foucart2009sparsest,MR2533469,van2009conditions,bertin11:_adapt_dantiz,juditsky2011accuracy,de2012remark} or \cite{CGLP,foucart2013mathematical}.

\subsection{Main Result: a Toolbox to get SRSR from Deviation Inequalities}
In this paper, we want to guarantee SRSR when the matrix $\M$ is chosen at random and $\hat x$ is produced by an $\ell_1$-minimization estimate such as \eqref{eq:L1minNoise}. 

The first random model we assume for $\M$ is as follows. Let~$R$ be a subset of $\{1,\ldots,p\}$ of size~$r$ and denote $\M_R$ the $(n\times r)$ matrix obtained by keeping the columns of~$\M$ that belongs to~$R$ and~$\M_R^\star$ its hermitian adjoint. Assume that all the Gram matrices $\M_R^\star \M_R$ satisfies
\eq
\label{covariance_model}
\forall r\in\{1,\ldots,\lceil 0.1230\, n\rceil \},\ 
\forall R\subseteq\{1,\ldots,p\}\ \mathrm{s.t.}\ |R|=r,\quad \M_R^\star \M_R\sim \C_{r,n}\,,
\qe
namely $\M_R^\star \M_R$ are identically distributed with respect to the law of the random covariance matrix~$\C_{r,n}$ of size $(r\times r)$ that may depend on the parameter $n$ also. As we will see in the sequel, standard models of covariance matrices $\C_{r,n}$ are given by $\C_{r,n}=({1}/{n})\,\X\X^\star$ where $\X\in\R^{r\times n}$ has {\it iid} entries drawn with respect to a law~$\mathcal L$. Note that we do not require that~$\M$ has independent entries here, we only assume that the matrices $\M_R\M_R^\star$ are identically distributed and that we have access to the rate function of the deviations on its extreme eigenvalues. 

Guaranteeing SRSR with sparsity parameter $s$ requires to bound the eigenvalues of $\M_R^\star \M_R$ when $r=2s$, see for instance Property~\eqref{eq:Condition_SRSR}. We consider the asymptotic proportional growth model where $s/n\to\rho$, {\it i.e.} we assume that size of the sparse vectors over number of equations tends to a constant. We also consider the parameter $\brho=r/n\in(0,0.1230)$ and since $r=2s$ we remark that $\bar \rho=2\rho$. 
We assume that we have access to a deviation inequality on extreme eigenvalues of~$\C_{r,n}$ with rate function $t\mapsto\mathds W(\brho,t)$ depending on this ratio. For instance, we will consider that for all $n\geq n_0(\bar \rho)$, 
\eq
\label{eq:deviation}
\forall t\in[0,0.6247),\quad
\Prob\big\{\big(\lambda_1-(1+\sqrt{\brho})^2\big)\vee\big((1-\sqrt{\brho})^2-\lambda_r\big)\geq t\big\}\leq c(\brho)e^{-n\mathds W(\brho,t)}
\qe
where  $n_0(\brho)\geq 2$ and $c(\brho)>0$ may depend on the ratio $\brho$, the function $t\mapsto\mathds W(\brho,t)$ is continuous and increasing on $[0,0.6247)$ such that $\mathds W(\brho,0)=0$, $\lambda_1$ and $\lambda_r$ are respectively the largest and the least eigenvalues of $\C_{r,n}$. Remark that the ‘‘bulk'' bounds $(1\pm \sqrt{\brho})^2$ are prescribed by the Marchenko-Pastur law. Consider the asymptotic proportional growth model for $\M$ so that $n/p\to\delta$, {\it i.e.} the number of equations over number of unknowns tends also to a constant. The main contribution of this paper is to give a bound on $(\delta,\rho)$ from the rate function $\mathds W$ so that SRSR holds with overwhelming probability. 

The second random model covered by this paper is given by matrices $\M$ satisfying the following property
\eq
\label{singular_model}
\forall r\in\{1,\ldots,\lceil 0.1230\,n\rceil \},\ 
\forall R\subseteq\{1,\ldots,p\}\ \mathrm{s.t.}\ |R|=r,\quad \M_R\sim \X_{r,n}\,,
\qe
namely $\M_R$ are identically distributed with respect to the law of the random rectangular matrix~$\X_{r,n}$ of size $(n\times r)$. Standard models of rectangular matrices $\X_{r,n}$ are given by $\X_{r,n}=({1}/{\sqrt n})\,\X^\star$ where $\X\in\R^{r\times n}$ has {\it iid} entries drawn with respect to a law~$\mathcal L$. As in the first model, we do not require that $\M$ has independent entries here, we only assume that the matrices $\M_R$ are identically distributed. In this case, we assume that we have a deviation inequality with rate function $\mathds W$ on the extreme singular values of $\X_{r,n}$. Again, our result gives a bound on~$(\delta,\rho)$ from the rate function $\mathds W$ so that SRSR holds with overwhelming probability. 

\begin{figure}[t]
\includegraphics[width=0.5\textwidth]{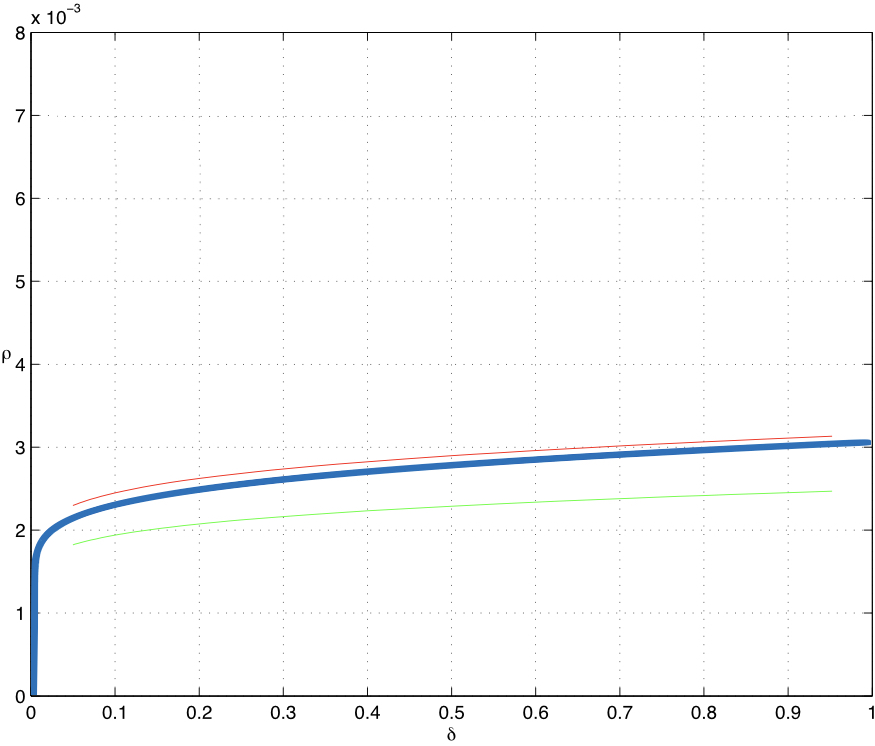}
\caption{{\bf [Bounds on SRSR in the {\it iid} Gaussian case]} The region below the curves gives pairs of $(\delta=n/p,\rho=s/n)$ for which SRSR holds with overwhelming probability when $\M$ has {\it iid} Gaussian entries. Our new bound \eqref{eq:lowerbound} in blue is comparable to the one of \cite{blanchard2011compressed} in red (derived using Foucart and Lai condition \cite[Theorem~2.1]{foucart2009sparsest}) and the one of~\cite{MR2412803} in green (derived from symmetric RICs bounds as in \eqref{eq:ConditionRIP}). This figure is an update of Figure 3.2 in \cite{blanchard2011compressed}.}
\label{fig:compa1}
\end{figure}

The detailed results are presented in Section~\ref{sec:main} and we give here the bound one can get in the case where~$\M$ has {\it iid} Gaussian entries. More precisely, we establish a new sufficient condition on SRSR that offers the same lower bound as previous state-of-the-art results such as the results presented in \cite{blanchard2011compressed}. Indeed, using Davidson-Szarek's deviation \cite{davidson2001local}, we prove that if
\eq
\label{eq:lowerbound}
\delta > \frac{1}{2\rho}\exp\Bigg[1-\frac{1}{4\rho}\bigg(\sqrt{\frac{{33 - 5 \sqrt{41}}}{8}}-\sqrt{2\rho}\bigg)^2 \Bigg]\,,
\qe
then SRSR holds with overwhelming probability when $\M$ has {\it iid} Gaussian entries, see Section~\ref{sqec:DS}. This bound is comparable to previous state-of-the-art result \cite{blanchard2011compressed,MR2417886}, see Figure~\ref{fig:compa1}.


\subsection{Byproduct Result: New Bounds on the Restricted Isometry Constants}
One property for assessing SRSR is the Restricted Isometry Property \cite{MR2236170,MR2300700} of order $s$ and parameter $c$, referred to as $\mathrm{RIP}(s,c,\M)$. It is defined by
\[
\forall x\in\Sigma_{s},\quad(1-c)\lVert x\lVert_{2}^{2}\leq\lVert\M x\lVert_{2}^{2}\leq(1+c)\lVert x\lVert_{2}^{2}.
\]
 Denote by $\mathfrak{c}(s,\M)$ the minimum of such $c$'s. One can prove (see Theorem 6.12 in \cite{foucart2013mathematical} for instance) that, if $\mathrm{RIP}$ such that
\eq
\label{eq:ConditionRIP}
\tag{FR-$\mathfrak c({2s})$}
\mathfrak{c}(2s,\M)<{4}\big/{\sqrt{41}}\simeq0.625\,,
\qe
holds and $\hat x$ is any solution to \eqref{eq:L1minNoise} then SRSR of order $s$ holds with $C,D>0$ depending only on~$\mathfrak{c}(2s,\M)$. A slightly modified RIP was introduced by Foucart and Lai in \cite{foucart2009sparsest,blanchard2011compressed}. They introduce two constants, called Restricted Isometry Constants~(RICs). For a matrix $\M$ of size $(n\times p)$, the RICs, $\cmin(s,\M)$ and $\cmax(s,\M)$, are defined as
\begin{align*}
\cmin&:=\min_{c_-\geq0}c_-\quad \mathrm{subject\ to}\quad(1-c_-)\lVert x\lVert_{2}^{2}\leq\lVert\M x\lVert_{2}^{2}\quad \mathrm{for\ all}\quad x\in\Sigma_{s},\\
\cmax&:=\min_{c_+\geq0}c_+\quad \mathrm{subject\ to}\quad(1+c_+)\lVert x\lVert_{2}^{2}\geq\lVert\M x\lVert_{2}^{2}\quad \mathrm{for\ all}\quad x\in\Sigma_{s}.
\end{align*}
Hence, it holds $(1-\cmin)\lVert x\lVert_{2}^{2}\leq\lVert\M x\lVert_{2}^{2}\leq (1+\cmax)\lVert x\lVert_{2}^{2}$ for all $x\in\Sigma_{s}$, where we recall that~$\Sigma_{s}$ denotes the set of vectors with at most $s$ nonzero coordinates. Reporting the influence of both extreme eigenvalues of covariance matrices built from $r=2s$ columns of $\M$, one can weaken~\eqref{eq:ConditionRIP}, see for instance Theorem 2.1 in~\cite{foucart2009sparsest}. Revisiting \cite{foucart2009sparsest} and \cite[Proof of Theorem 6.13 (page 145)]{foucart2013mathematical}, this paper provides the weakest condition to get SRSR in the following theorem, see Appendix~\ref{app:proof_Condition_SRSR} for a proof.

\begin{thm}
\label{thm:Condition_SRSR}
If $\M$ satisfies this asymmetric Restricted Isometry Property with RICs such that
\eq
\label{eq:Condition_SRSR} 
\tag{SRSR-$\gamma(2s)$}
\gamma(2s,n,p):=\frac{1+\cmax(2s,\M)}{1-\cmin(2s,\M)}<\frac{(4+\sqrt{41})^2}{25}\simeq4.329,
\qe
then the Stable and Robust Sparse Recovery (SRSR) property of order $s$ holds with positive constants $C$ and $D$ depending only on $\cmin(2s,\M)$ and $\cmax(2s,\M)$.
\end{thm}
\begin{rem}
The condition to get SRSR described in~\cite[Theorem~2.1]{foucart2009sparsest} can be equivalently written as $\gamma(2s,n,p)<(5+\sqrt 2)/(1+\sqrt 2)\simeq2.657$ which is a stronger requirement than Condition \eqref{eq:Condition_SRSR}. Also, remark that Condition~\eqref{eq:Condition_SRSR} leads to the inequality $(1+\mathfrak{c}(2s,\M))/(1-\mathfrak{c}(2s,\M))<(4+\sqrt{41})^2/25$ and one can check that this is exactly Condition~\eqref{eq:ConditionRIP}. From this remark, one can view \eqref{eq:Condition_SRSR} as a generalization of~\eqref{eq:ConditionRIP} to the frame of asymmetric isometry constants.
\end{rem}

Note that RIP constants are involved in other frameworks such as low-rank or group-sparse recovery, see for instance \cite{traonmilin2016stable,cai2014sparse,ayaz2016uniform}. Our results do not directly apply to those frameworks but the control on RIP constants given by the methodology presented in this paper (namely \eqref{eq:bounds_RIC_eigen_min}, \eqref{eq:bounds_RIC_eigen_max}, \eqref{eq:bound_RIC_singular_min} and \eqref{eq:bound_RIC_singular_max}) may be invoked in these settings.

\subsection{From rate functions to SRSR}\label{sec:devtoRIC}
In this paper, we provide a toolbox to derive upper bounds on~RICs (with overwhelming probability) from deviation inequalities on extreme eigenvalues (or extreme singular values) of covariance matrices $\C_{r,n}$. 
{The known asymptotic behavior of these extreme eigenvalues provides an expected behavior for $\mathds W(\brho,t)$ in both variables $t$ and~$\brho$.} Notably, it appears along our analysis that bounds on SRSR and RICs are extremely dependent on the behavior, for fixed $t$, of the rate function $\brho\mapsto\mathds W(\brho,t)$ when $\brho$ is small, and possibly tending to zero. {More details will be given in Section~\ref{sec:state_art_dev}.}
Unfortunately, this dependence is sometimes unclear in the literature and we have to take another look at state-of-the-art results in this field. Revisiting the paper of Feldheim and Sodin~\cite{feldheim2010universality} on sub-Gaussian matrices, Appendix~\ref{sec:dev_Rad} reveals the dependency on $\brho$ as well as bounds on the constant appearing in their rate function~$\mathds W_{\mathrm{FS}}$ for the special case of Rademacher entries. Other important results due to Ledoux and Rider \cite{ledoux2010small}, and Davidson and Szarek~\cite{davidson2001local} are investigated in Appendix~\ref{sec:keylem}. 



\subsection{Previous works on bounding RIP and RICs}
\label{sec:PreviouslyOnHBO}
The existence of RIP matrices with bounded RIP constant such as \eqref{eq:ConditionRIP} has been proved using random matrix models, see \cite{MR2453366,MR2453368,adamczak2011restricted,CGLP} for instance. 
It has been proved that \eqref{eq:ConditionRIP} holds with overwhelming probability for a large class of random matrix models as soon as the interplay between sparsity~$s$, number of measurements $n$ and number of unknown parameters~$p$ satisfies 
\eq
\label{eq:sLOGps}
n\geq c_1\, s\log(c_2p/s)
\qe
for some universal constants $c_1$ and $c_2$ (that might depend on the random matrix model). It should be mentioned that finding deterministic matrices satisfying \eqref{eq:ConditionRIP} with $n=\mathcal O( s\log(p/s))$ is one of the most prominent open problem in Compressed Sensing, see \cite{foucart2013mathematical} for instance. Furthermore, it has been shown in \cite[Proposition~2.2.17]{CGLP} that the converse is true for any matrix $\M$. If the SRSR recovery \eqref{eq:SRSR1} or \eqref{eq:SRSR2} (with $\eta=0$) holds then necessarily $n\geq c'_1\, s\log(c'_2p/s)$ for some universal constants $c'_1$ and $c'_2$. Since we have lower and upper bounds of the same flavor, it seems that the condition \eqref{eq:sLOGps} captures all we need to know about $\ell_1$-recovery schemes. In reality, there is a gap between the constants appearing in the upper and lower bounds. A simple way to witness it is to consider the companion problem when there is no additive errors. 
In this case $\mathbf e=0$ in \eqref{eq:LinModelNoise}
and $\eta=0$ in \eqref{eq:L1minNoise}, then stable recovery occurs for all target vector~$x_0$ if and only a property called \textquotedblleft Null-Space Property\textquotedblright\  (NSP) holds. As for RIP, one can prove that \eqref{eq:sLOGps} depicts a necessary and sufficient condition on NSP up to a change of constants, see for instance \cite{CGLP,azais2014rice}. Nevertheless, a better description of this property is offered in the works \cite{donoho2005neighborliness,donoho2009counting,donoho2009observed} since the authors provide a phase transition on NSP for large Gaussian matrices with {\it iid} entries. Let us also mention the important papers \cite{mccoy2012sharp,amelunxen2013living} that give quantitative estimates of ``weak'' thresholds appearing in convex optimization, including the location and the width of the transition region for~NSP.

Following this outbreaking result, one can wonder whether a phase transition holds for properties guaranteeing SRSR such as Condition \eqref{eq:ConditionRIP} or the asymmetric~\eqref{eq:Condition_SRSR}. To the best of our knowledge, the first work looking for a phase transition on SRSR can be found in \cite{blanchard2011compressed} where the authors considered matrices with independent standard Gaussian entries and used an upper bound on the joint density of the eigenvalues to derive a region where~\eqref{eq:Condition_SRSR} holds. Their lower bound is not explicit but one can witness in \cite[page 119]{blanchard2011compressed}. {Furthermore they provide web forms for the calculation of bounds on RICs, which are available at \href{https://people.maths.ox.ac.uk/tanner/ric_bounds.shtml}{Jared Tanner's} webpage. 
Shortly after, Bah and Tanner improved these bounds in \cite{bah2010improved} by preventing the use of union bound over all sub-matrices built from $r=2s$ columns of~$\M$ by grouping those which share a substantial number of columns. Their 
bounds are still implicit but web forms for their calculation are available at the same place. The same authors provided later~\cite{bah2014asymptotics} explicit bounds for the RICs in extreme asymptotic regime:
\begin{enumerate}[(a)]
 \item \label{item:rho_0_delta_fixed} when $\rho \to 0$ and $\delta>0$ is fixed,
 \item \label{item:rho_fixed_delta_0} when $\delta \to 0$ and $\rho>0$ is fixed,
 \item \label{item:rho_0_delta_0}when $\rho=\frac{-1}{\gamma \log \delta}$ ($\gamma$ is a fixed parameter) and $\delta \to 0$.
\end{enumerate}}
\noindent
In the sequel, we may refer to these regimes as Regime $(a)$, $(b)$ and $(c)$ respectively.

We would like to point out recent ‘‘off-the-shelf'' concentration results, often based on generic chaining, that imply bounds on RIP constants, at least for subgaussian random matrices. The reader may consult for instance \cite[Section 5]{dirksen2015tail} or \cite{liaw2017simple}. From these results, classical bounds on gaussian widths of sparse vectors lead to null space properties or RIP constants, see \cite[Section~2.6]{liaw2017simple}. We would also point out the interesting results on the ‘‘Small Ball Method'', see \cite{lecue2014sparse,lecue2018regularization} for instance, that requires only weak moment assumptions to get sparse recovery guarantees.

\subsection{Outline}
The paper is organized as follows. Section \ref{sec:main} states the main results: it provides a general method to derive bounds on RICs and phase transition in Condition \eqref{eq:Condition_SRSR} from deviation inequalities on eigenvalues or singular values. Subsection \ref{sec:state_art_dev} begins with a discussion on what is expected for such deviation inequalities. The general method described in Subsections \ref{sec:dev_eigen_to_RIC} and \ref{sec:dev_singular_to_RIC} is then applied to previously known inequalities. Section \ref{sec:main} ends with a summary of the obtained bounds. 

The proofs are contained in the appendix. Appendix \ref{app:proof_Condition_SRSR} provides the proof of Theorem \ref{thm:Condition_SRSR}, while Appendix \ref{sec:keylem} and Appendix \ref{proof:key_singular} contain the proofs of Theorems~\ref{thm:key_eigen} and \ref{thm:key_singular}. In Appendix~\ref{sec:dev_Rad}, we follow the steps of \cite{feldheim2010universality} to provide an upper bound on the constant in the deviation inequality for extreme singular values of Rademacher matrices.


\subsection*{Acknowledgments.}
We would like to thank Sasha Sodin for his patient answers to our many questions. Moreover, this paper greatly benefited from the comments of anonymous referees on previous version of this paper.

\pagebreak[3]

\section{From deviations to RICs and SRSR bounds}
\label{sec:main}
Following the framework of~\cite{blanchard2011compressed}, we provide asymptotic bounds on RICs in the proportional growth model. {As previously explained, we suppose that we are able to control the deviation of extreme eigenvalues or singular values.} We aim at controlling uniformly the extreme eigenvalues, the combinatorial complexity is standardly~\cite{blanchard2011compressed} controlled by the quantity $\delta^{-1}\H(\brho\delta)$ where 
\[
\H(t):=-t\log t-(1-t)\log(1-t)\quad \textrm{for}\ t \in (0,1)\,,
\]
denotes the Shannon entropy. {The improvement introduced in \cite{bah2010improved} to deal with this combinatorial complexity could be used here but we chose not to do so as it would have turned our explicit bounds into implicit ones.} One may remark that the  quantity 
\[
t_0:=\mathds W^{-1}\Big(\brho,{\delta}^{-1}\H(\brho\delta)\Big)
\] 
governs the order of the deviation in the rate function $t\mapsto\mathds W(\brho,t)$ when bounding the extreme eigenvalues uniformly over all possible supports $S$ of size $r$ among the set of indices $\{1,\ldots,p\}$, see~\eqref{Psi_cov} and \eqref{Psi_rect} in the functions~$\Psi_{\min/\max}$. {Here $\mathds W^{-1}(\brho,.)$ denotes the inverse of~$\mathds W$ with respect to its second variable.} 

The next theorems gives the probability that the matrix $\M$ satisfies \eqref{eq:Condition_SRSR}. From Theorem~\ref{thm:Condition_SRSR} note that this event is included in the event such that SRSR holds, namely
\[
\big\{\M\ \mathrm{satisfies}\ \eqref{eq:Condition_SRSR}\ \mathrm{with}\ s\leq s_0\big\}
\subset\big\{\mathrm{SRSR\ holds\ with\ } s\leq s_0\big\}
\]
where the right hand event means that for any $x_0\in\mathds R^p$, any solution $\hat x$ to \eqref{eq:L1minNoise} satisfies SRSR for all parameters $s$ such that $s\leq s_0$.

\subsection{Using extreme eigenvalues deviations inequalities}
\label{sec:dev_eigen_to_RIC}
Useful constants are  $\rho_0:=(33 - 5 \sqrt{41})/16\simeq0.0615$ and $\sqrt{\rho_0}\simeq0.3508$. Also, we denote $\tau_0:=4/\sqrt{41}\simeq0.6247$. The key main result is the following theorem proved in Section~\ref{sec:keylem}.

\begin{thm}
\label{thm:key_eigen}
Assume that for all $0<\bar{\rho}<2\rho_0$, the largest eigenvalue $\lambda_1$ and the smallest eigenvalue~$\lambda_r$ of a $(r\times r)$ covariance matrix $\C_{r,n}$ 
where $r:=\lfloor \brho n\rfloor$ 
satisfy for all $n\geq n_0(\brho)$,
\[
\forall 0\leq  t<\tau_0,\quad
\Prob\Big\{\big(\lambda_1-(1+\sqrt{\brho})^2\big)\vee\big((1-\sqrt{\brho})^2-\lambda_r\big)\geq t\Big\}\leq c(\brho)e^{-n\mathds W(\brho,t)}
\]
where $n_0(\brho)\geq 2$ and $c(\brho)>0$ may both depend on $\brho$, the function $t\mapsto\mathds W(\brho,t)$ is continuous and increasing on $[0,\tau_0)$ such that $\mathds W(\brho,0)=0$. Then for any $0<\delta<1$ and $0<\rho<\rho_0$ 
such that
\eq
\label{eq:CondLemma}
\delta>\Psi_0^{(1)}(\rho,\mathds W):=\frac{1}{2\rho}\exp\Big\{1-\frac{1}{2\rho}{\mathds W\Big[2\rho,2\tau_0(\sqrt{\rho}-\sqrt{\rho_0})\big(\sqrt\rho-\frac{1}{2\sqrt{\rho_0}}\big)\Big]}\Big\}\,,
\qe
and for any sequence of $(n\times p)$ matrices $(\M^{(n)})_{n\geq2}$ with $n/p\to\delta$ satisfying \eqref{covariance_model}, it holds that
\[
\Prob\Big\{{\M^{(n)}}\ \mathrm{satisfies}\ \eqref{eq:Condition_SRSR}\ \mathrm{with}\ s\leq s_0=\lfloor \rho n\rfloor\Big\}
\geq1-2c(2\rho)e^{-nD_1(\delta,\rho)}
\to1
\]
for some $D_1(\delta,\rho)>0$ that may depend on $\delta$ and $\rho$.

Furthermore, for all $\varepsilon>0$ and for all $\brho$ and $\delta$ such that ${\delta}^{-1}\H(\brho\delta)$ belongs to the range of~$\mathds W(\brho,\cdot)$, it holds
\begin{align}\label{eq:bounds_RIC_eigen_min}
\P\Big\{\cmin(\lfloor \bar\rho n\rfloor,\M^{(n)}) \geqslant \Psi_{\min}^{(1)}(\delta,\bar\rho,\mathds W) +\varepsilon\Big\} &\leqslant c(\bar\rho)e^{-nD_2(\bar\rho,\delta,\varepsilon)},\\
\label{eq:bounds_RIC_eigen_max}
\P\Big\{\cmax(\lfloor \bar\rho n\rfloor,\M^{(n)}) \geqslant\Psi_{\max}^{(1)}(\delta,\bar\rho,\mathds W)+\varepsilon\Big\} &\leqslant c(\bar\rho)e^{-nD_2(\bar\rho,\delta,\varepsilon)},
\end{align}
where $D_2(\bar\rho,\delta,\varepsilon)>0$ and we denote 
\begin{align}
\label{Psi_cov}
\Psi_{\min}^{(1)}(\delta,\bar\rho,\mathds W)&:=\sqrt{\bar\rho}(2-\sqrt{\bar\rho})+t_0\,,\\
\notag
\Psi_{\max}^{(1)}(\delta,\bar\rho,\mathds W)&:= \sqrt{\bar\rho}(2+\sqrt{\bar\rho})+t_0\,,
\end{align}
with $t_0:=\mathds W^{-1}(\bar\rho,{\delta}^{-1}\H(\bar\rho\delta))$.
\end{thm}

\subsection{Using extreme singular values deviations}
\label{sec:dev_singular_to_RIC}
A similar result can be derived from deviations on singular values, a proof is given in Section~\ref{proof:key_singular}. 

\begin{thm}
\label{thm:key_singular}
Assume that for all $0<\bar\rho<2\rho_0$, the largest singular value  $\sigma_1$ and the smallest singular value $\sigma_r$ of a $(n\times r)$ matrix $\X_{r,n}$ where $r:=\lfloor \rho n\rfloor$ satisfy for all $n\geq n_0(\bar\rho)$,
\[
\forall 0<t<\sqrt{2\rho_0},\quad
\Prob\Big\{\big({\sigma_1}-(1+\sqrt{\bar\rho})\big)\vee\big((1-\sqrt{\bar\rho})-{\sigma_r}\big)\geq t\Big\}\leq c(\bar\rho)e^{-n\mathds W(\bar\rho,t)}
\]
where $n_0(\bar\rho)\geq 2$ and $c(\bar\rho)>0$ may both depend on $\bar\rho$, the function $t\mapsto\mathds W(\bar\rho,t)$ is continuous and increasing on $[0,\sqrt{2\rho_0})$ such that $\mathds W(\bar\rho,0)=0$. Then for any $0<\delta<1$ and $0<\rho<\rho_0$ 
such that
\eq
\label{eq:CondLemma2}
\delta>\Psi_0^{(2)}(\rho,\mathds W):={\frac{1}{2\rho}}\exp\Big\{1-\frac{1}{2\rho}{\mathds W\big[2\rho,\sqrt{2\rho_0}-\sqrt{2\rho}\big]}\Big\}\,,
\qe
and for any sequence of $(n\times p)$ matrices $(\M^{(n)})_{n\geq2}$ with $n/p\to\delta$ satisfying \eqref{singular_model}, it holds that
\[
\Prob\Big\{{\M^{(n)}}\ \mathrm{satisfies}\ \eqref{eq:Condition_SRSR}\ \mathrm{with}\ s\leq s_0=\lfloor \rho n\rfloor\Big\}
\geq1-2c(2\rho)e^{-nD_1(\delta,\rho)}
\to1
\]
for some $D_1(\delta,\rho)>0$ that may depend on $\delta$ and $\rho$.

Furthermore, for all $\varepsilon>0$ and for all $\bar\rho$ and $\delta$ such that ${\delta}^{-1}\H(\bar\rho\delta)$ belongs to the range of~$\mathds W(\bar\rho,\cdot)$, it holds
\begin{align}
\label{eq:bound_RIC_singular_min}\P\Big\{\cmin(\lfloor \bar\rho n\rfloor,\M^{(n)}) \geqslant \Psi_{\min}^{(2)}(\delta,\bar\rho,\mathds W) +\varepsilon\Big\} &\leqslant c(\bar\rho)e^{-nD_2(\bar\rho,\delta,\varepsilon)},\\
\label{eq:bound_RIC_singular_max}\P\Big\{\cmax(\lfloor \bar\rho n\rfloor,\M^{(n)}) \geqslant\Psi_{\max}^{(2)}(\delta,\bar\rho,\mathds W)+\varepsilon\Big\} &\leqslant c(\bar\rho)e^{-nD_2(\bar\rho,\delta,\varepsilon)},
\end{align}
where $D_2(\bar\rho,\delta,\varepsilon)>0$ and we denote 
\begin{align}
\label{Psi_rect}
\Psi_{\min}^{(2)}(\delta,\bar\rho,\mathds W)&:=\min\{1,(\sqrt{\bar\rho}+t_0)(2-\sqrt{\bar\rho}-t_0)\}\,,\\
\notag
\Psi_{\max}^{(2)}(\delta,\bar\rho,\mathds W)&:=(\sqrt{\bar\rho}+t_0)(2+\sqrt{\bar\rho}+t_0)\,,
\end{align}
with $t_0:=\mathds W^{-1}(\bar\rho,{\delta}^{-1}\H(\bar\rho\delta))$. 
\end{thm}

{Theorems \ref{thm:key_eigen} and \ref{thm:key_singular} give a general method to derive bounds on RICs from deviation inequalities satisfied by the extreme eigenvalues or singular values of a random matrix. In the following subsection, three known deviation inequalities are used to provide such bounds for Gaussian and Rademacher matrices.}

\subsection{State-Of-The-Art deviation inequalities}\label{sec:state_art_dev}
{The asymptotic behavior of extreme eigenvalues of random covariance matrices with iid entries has been known for some years. From this behavior and the concentration of measure phenomenon, we present what is expected for deviation inequalities for extreme eigenvalues of such matrices with sub-Gaussian entries. This is what we call ``ideal deviations''. The next two paragraphs are devoted to deviation inequalities for Gaussian matrices due to Davidson and Szarek \cite{davidson2001local}, and Ledoux and Rider~\cite{ledoux2010small}. The last paragraph focuses on a deviation inequality for Rademacher matrices, proved by Feldheim and Sodin~\cite{feldheim2010universality}.}

In what follows, the bounds on RICs are written in terms of parameters $\bar\rho=\frac{r}{n}$ and $\delta=\frac{n}{p}$. When expressing Condition \eqref{eq:CondLemma} or \eqref{eq:CondLemma2}, we use parameter $\rho=\frac{s}{n}$, with $s=r/2$. Note that throughout this paper, it holds that $\bar\rho=2\rho$.

\subsubsection{Ideal deviations}
\label{sec:ideal_dev}
The asymptotic behavior of extreme eigenvalues for random covariance matrices was first established for matrices with Gaussian entries \cite{Jo_2000_largest_eigenvalue_complex,BoFo_2003_hard_soft_edge_transition} and extended to ones with more general entries in \cite{soshnikov_2002_largest_cov,peche_2009_largest_cov,feldheim2010universality,pillai_yin_2014_univ_cov,wang_2012_cov_edge}.
The largest eigenvalue fluctuations are described by the following:
\eq
\notag
\bigg[\frac{n\brho^{1/4}}{ (1+\sqrt{\brho})^2}\bigg]^{\frac23}(\lambda_1-(1+\sqrt{\brho})^2) \overset{(d)}{\underset{n \to \infty}{\to}}
F_1,
\qe
where $F_1$ is the so-called Tracy-Widom law. As for the smallest eigenvalue, when $\brho<1$ (which is true in our setting),
\eq
\notag
\bigg[\frac{n\brho^{1/4}}{(1-\sqrt{\brho})^2}\bigg]^{\frac23}({(1-\sqrt{\brho})^2-\lambda_r}) \overset{(d)}{\underset{n \to \infty}{\to}}
F_1.
\qe

We focus on the largest eigenvalue $\lambda_1$ and write:
\begin{align*}
 \P(\lambda_1\geqslant (1+\sqrt{\brho})^2+t) & = \P\bigg(\Big[\frac{n\brho^{1/4}}{ (1+\sqrt{\brho})^2}\Big]^{\frac23}(\lambda_1-(1+\sqrt{\brho})^2)\geqslant \frac{n^{2/3}\brho^{1/6}}{(1+\sqrt{\brho})^{4/3}}t\bigg).
\end{align*}
This deviation probability is therefore expected to be close to \[1-F_1\bigg(\frac{n^{2/3}\brho^{1/6}}{(1+\sqrt{\brho})^{4/3}}t\bigg),\] where $F_1$ is the cdf of the Tracy-Widom distribution. Thus it is expected to be close to the tail behavior of $F_1$ at $\infty$, which is actually known:
\[ 1-F_1(x) \underset{x \to \infty}{\sim} e^{-\frac23 x^{3/2}}.\]
As a consequence, deviation inequalities for the largest eigenvalue are expected to conform to
\begin{equation}
\notag
 \P\Big(\lambda_1\geqslant (1+\sqrt{\brho})^2+t\Big) \leqslant C\exp\Big(-c\frac{\brho^{1/4}}{(1+\sqrt{\brho})^2}nt^{3/2}\Big),
\end{equation}
at least for $t$ of the order of the spectrum width (which behaves asymptotically as $\mathcal{O}(\sqrt{\brho})$). 
For bigger $t$, due to the concentration of measure phenomenon, the expected behavior is the following:
\begin{equation}
\notag
\P\Big(\lambda_1\geqslant (1+\sqrt{\brho})^2+t\Big) \leqslant Ce^{-cn\min(t,t^2)}. 
\end{equation}
Similar results should hold for the smallest eigenvalue, except that $\lambda_r\geqslant 0$ almost surely and therefore only moderate deviations can occur. See \cite{ledoux_deviation_inequalities} for a detailed survey on this subject and \cite[Page 1322]{ledoux2010small} for a specific discussion on the change of behavior occurring around $t=\mathcal{O}(\sqrt{\brho})$.

Considering these expected deviation inequalities, it may be possible to prove the following for sub-Gaussian random matrices.
\[
\forall t>0,\quad
\Prob\Big\{\lambda_1-(1+\sqrt{\brho})^2\geq t\Big\}\leq c(\brho)e^{-n\mathds W_{\mathrm{TW}}(\brho,t)}
\]
where $\lambda_1$ denotes the largest eigenvalue of a $(r\times r)$ covariance matrix $\C_{r,n}$ with {\it iid} sub-Gaussian entries and 
\begin{align}\label{eq:ideal_rate}
\mathds W_{\mathrm{TW}}(\brho,t)&:=\frac{1}{C}\bigg\{\frac{\brho^{\frac14}}{(1+\sqrt{\brho})^2} t^{\frac32}\mathbbm{1}_{t\leqslant \sqrt{\brho}}+ \frac{t^2}{(1+\sqrt{\brho})^2} \mathbbm{1}_{\sqrt{\brho}<t\leqslant 1}+\frac{t}{(1+\sqrt{\brho})^2} \mathbbm{1}_{t>1}\bigg\}\,,
\end{align}
where $C>0$. A similar deviation inequality may be established for the smallest eigenvalue~$\lambda_r$ with almost a similar $\mathds W$ function (the $(1+\sqrt{\brho})^2$ terms should be replaced by $(1-\sqrt{\brho})^2$).  
We should obtain
\begin{align*}
\mathds W_{\mathrm{TW}}^{-1}(\brho,u)& =
\underbrace{\frac{C^{2/3}(1+\sqrt{\brho})^{4/3}}{\brho^{\frac16}} u^{\frac23}}_{\text{small deviation part}}
\mathbbm{1}_{u\leqslant \frac{\brho}{C(1+\sqrt{\brho})^2}}
\\
& \quad + \quad
\underbrace{ C^{1/2}(1+\sqrt{\brho})\sqrt{u}}_{\text{moderate deviation part}} 
\mathbbm{1}_{\frac{\brho}{C(1+\sqrt{\brho})^2}<u\leqslant \frac{1}{C(1+\sqrt{\brho})^2}}
+ 
\underbrace{C(1+\sqrt{\brho})^2u}_{\text{large deviation part}} \mathbbm{1}_{u>\frac{1}{C(1+\sqrt{\brho})^2}}.
\end{align*}
Theorem \ref{thm:key_eigen} could then be invoked to get bounds on RICs,
\begin{align*}
 \cmax & \leqslant \sqrt{\brho}(2+\sqrt{\brho})+t_0,\\
 \cmin & \leqslant \sqrt{\brho}(2-\sqrt{\brho})+t_0,
\end{align*}
where $t_0=\mathds W_{\mathrm{TW}}^{-1}\big(\brho,\frac{1}{\delta}\H(\brho\delta)\big)$. 
\begin{rem}
Note that in the two asymptotic Regimes \eqref{item:rho_0_delta_fixed} and~\eqref{item:rho_0_delta_0} of~\cite{
bah2014asymptotics}, $\frac{1}{\delta}\H(\brho\delta)$ will be larger than $\frac{\brho}{C(1+\sqrt{\brho})^2}$. Therefore it seems that the most important part in the rate function~\eqref{eq:ideal_rate} for our present use is the moderate and large deviation parts, arising from the concentration of measure phenomenon.
\end{rem}

\subsubsection{Davidson and Szarek's deviations}
\label{sqec:DS}
Consider a $(r\times n)$ matrix $\X$ with {\it iid} standard Gaussian entries. In the paper \cite{davidson2001local}, Davidson and Szarek have shown that for all $0<\brho<1$ it holds
\[
\forall t>0,\quad
\Prob\Big\{\Big(\frac{\sigma_1(\X)}{\sqrt{n}}-(1+\sqrt{\brho})\Big)\vee\Big((1-\sqrt{\brho})-\frac{\sigma_r(\X)}{\sqrt{n}}\Big)\geq t\Big\}\leq 2e^{-n\mathds W_{\mathrm{DS}}(\brho,t)}
\]
where $\sigma_i(\X)$ denotes the singular values of $\X$ and $\mathds W_{\mathrm{DS}}(\brho,t):={t^2}/2$, see \cite[page 291]{foucart2013mathematical} for instance. This inequality relies on the concentration of measure phenomenon. Note that
\[\mathds W_{\mathrm{DS}}^{-1}(\brho,u)=\sqrt{2u}.\]
Theorem \ref{thm:key_singular} applied here gives the following high probability bounds on RICs
\begin{align*}
 \cmax & \leqslant (\sqrt{\brho}+t_0)(2+\sqrt{\brho}+t_0),\\
 \cmin & \leqslant (\sqrt{\brho}+t_0)(2-\sqrt{\brho}-t_0),
\end{align*}
where $t_0=\sqrt{\frac{2}{\delta}\H(\brho\delta)}$. 

\begin{rem}
In the three asymptotic Regimes \eqref{item:rho_0_delta_fixed}, \eqref{item:rho_fixed_delta_0} and \eqref{item:rho_0_delta_0}, these bounds on RICs behave similarly to the ones obtained by Bah and Tanner in \cite{bah2014asymptotics}, except that constants are better in \cite{bah2014asymptotics}. Note that this deviation has been used in the paper \cite[Lemma 3.1]{MR2243152} to bound the RIP constant. 
\end{rem}

Furthermore, Theorem \ref{thm:key_singular} states that Condition \eqref{eq:Condition_SRSR} is satisfied with high probability whenever
\[ \delta > \frac{1}{2\rho}\exp\Big[1-\frac{1}{2\rho}(\sqrt{\rho_0}-\sqrt{\rho})^2 \Big].\]
When $\rho$ is small (which is the case in the Regimes \eqref{item:rho_0_delta_fixed} and \eqref{item:rho_0_delta_0}), this condition approximately writes
\[ \delta > \frac{1}{2\rho}\exp\Big[-\frac{\rho_0}{2\rho}\Big].\]
\subsubsection{Ledoux and Rider's deviations} 
Ledoux and Rider proved in \cite{ledoux2010small} small deviation inequalities for $\beta$ Hermite and Laguerre Ensembles. Their work rely on the tridiagonal model for these matrix ensembles and on a variational formulation of the Tracy-Widom distribution. For real covariance matrices, their deviation inequality for the largest eigenvalue is the following. For all $0<\brho<1$ and for all $n\geq2$, setting $r=\lfloor \brho n\rfloor$,
\[
\forall t>0,\quad
\Prob\Big(\lambda_1-(1+\sqrt{\brho})^2\geq t\Big)\leq c(\brho)e^{-n\mathds W_{\mathrm{LR}}^{\max}(\brho,t)}
\]
where $\lambda_1$ denotes the largest eigenvalue of a $(r\times r)$ covariance matrix $\C_{r,n}$ with {\it iid} standard Gaussian entries and 
\begin{align*}
\mathds W_{\mathrm{LR}}^{\max}(\brho,t)&:=\frac{\brho\,^{\frac14}}{C_{\mathrm{LR}}(1+\sqrt{\brho})^3} t^{\frac32}\mathbbm{1}_{t\leqslant \sqrt{\brho}(1+\sqrt{\brho})^2}+ \frac{\brho\,^{\frac12}}{C_{\mathrm{LR}}(1+\sqrt{\brho})^2} t\mathbbm{1}_{t> \sqrt{\brho}(1+\sqrt{\brho})^2}\,,
\end{align*}
where $C_{\mathrm{LR}}>0$ may be bounded explicitly from \cite{ledoux2010small}.
As explained in Section~\ref{sec:devtoRIC}, the dependency of function $\mathds W$ in parameter $\brho$ is of crucial importance in our analysis. Therefore, we choose to write the most precise deviation inequalities the paper reached, even in the case when $r/n$ is bounded. For $\lambda_r$, we follow the procedure explained in \cite[Section 5, page 1338]{ledoux2010small} to write the following
\[
\forall t>0,\quad
\Prob\Big((1-\sqrt{\brho})^2-\lambda_r\geq t\Big)\leq c(\brho)e^{-n\mathds W_{\mathrm{LR}}^{\min}(\brho,t)}
\]
where
\begin{align*}
\mathds W_{\mathrm{LR}}^{\min}(\brho,t)&:=\frac{\brho^{\frac14}}{C_{\mathrm{LR}}(1-\sqrt{\brho})^3} t^{\frac32}\mathbbm{1}_{t\leqslant \sqrt{\brho}(1-\sqrt{\brho})^2}+ \frac{\brho^{\frac12}}{C_{\mathrm{LR}}(1-\sqrt{\brho})^2} t\mathbbm{1}_{t> \sqrt{\brho}(1-\sqrt{\brho})^2}\,.
\end{align*}
In order to simplify the analysis of the phase transition, observe that $\mathds W_{\mathrm{LR}}^{\max}(\brho,t)\leqslant \mathds W_{\mathrm{LR}}^{\min}(\brho,t)$ for all $\brho$ and $t$. This yields
\[
\forall t>0,\quad
\Prob\Big\{\big(\lambda_1-(1+\sqrt{\brho})^2\big)\vee\big((1-\sqrt{\brho})^2-\lambda_r\big)\geq t\Big\}\leq c(\brho)e^{-n\mathds W_{\mathrm{LR}}(\brho,t)}
\]
where 
\begin{align*}
\mathds W_{\mathrm{LR}}(\brho,t)&:=\mathds W_{\mathrm{LR}}^{\max}(\brho,t)\,.
\end{align*}
Therefore
\[\mathds W_{\mathrm{LR}}^{-1}(\brho,u)=C_{\mathrm{LR}}^{2/3}\frac{(1+\sqrt{\brho})^2}{\brho^{1/6}}u^{2/3}\mathbbm{1}_{u\leqslant \frac{\brho}{C_{\mathrm{LR}}}}+ C_{\mathrm{LR}}\frac{(1+\sqrt{\brho})^2}{\sqrt{\brho}}u\mathbbm{1}_{u> \frac{\brho}{C_{\mathrm{LR}}}}.\]
Theorem \ref{thm:key_eigen} applied here gives the following high probability bounds on RICs:
\begin{align*}
 \cmax & \leqslant \sqrt{\brho}(2+\sqrt{\brho})+t_0,\\
 \cmin & \leqslant \sqrt{\brho}(2-\sqrt{\brho})+t_0,
\end{align*}
where $t_0=\mathds W_{\mathrm{LR}}^{-1}(\brho,\frac{1}{\delta}\H(\brho\delta))$. 
\begin{rem}
In the three asymptotic Regimes \eqref{item:rho_0_delta_fixed}, \eqref{item:rho_fixed_delta_0} and \eqref{item:rho_0_delta_0}, it may be shown that $\frac{1}{\delta}\H(\brho\delta)\geqslant \frac{\brho}{C_{\mathrm{LR}}}$. These bounds on RICs behave similarly to the ones obtained by Bah and Tanner in \cite{bah2014asymptotics} in Regime~\eqref{item:rho_0_delta_fixed}, except that their constants are better. In Regimes~\eqref{item:rho_fixed_delta_0} and \eqref{item:rho_0_delta_0}, they behave badly compared to those of \cite{bah2014asymptotics}. 
\end{rem}

Furthermore, Theorem \ref{thm:key_eigen} states that Condition \eqref{eq:Condition_SRSR} is satisfied with high probability whenever
\[ \delta > \frac{1}{2\rho}\exp\Big\{1-\frac{1}{2\rho}\mathds W_{\mathrm{LR}}\Big[2\rho,2\tau_0(\sqrt{\rho}-\sqrt{\rho_0})\big(\sqrt{\rho}-\frac{1}{2\sqrt{\rho_0}}\big)\Big] \Big\}.\]
When $\rho$ is small (which is the case in Regimes \eqref{item:rho_0_delta_fixed} and \eqref{item:rho_0_delta_0}), the second argument in $\mathds W_{\mathrm{LR}}$ is approximately $\tau_0$ and this condition approximately writes
\[ \delta > \frac{1}{2\rho}\exp\Big[-\frac{\tau_0}{C_{\mathrm{LR}}\sqrt{2\rho}}\Big].\]

\subsubsection{Feldheim and Sodin's deviations}
For all $0<\brho<1$ and for all $n\geq n_0$, setting $r=\lfloor \brho n\rfloor$ it follows from \cite{feldheim2010universality} that
\[
\forall t>0,\quad
\Prob\Big\{\big(\lambda_1-(1+\sqrt{\brho})^2\big)\vee\big((1-\sqrt{\brho})^2-\lambda_r\big)\geq t\Big\}\leq c(\brho)e^{-n\mathds W_{\mathrm{FS}}(\brho,t)}
\]
where $\lambda_i$ denotes the eigenvalues of a $(r\times r)$ covariance matrix $\C_{r,n}$ with {\it iid} Rademacher entries and 
\begin{align*}
\mathds W_{\mathrm{FS}}(\brho,t)&:=\frac{\brho\log(1+\frac{t}{2\sqrt{\brho}})^{\frac32}}{C_{\mathrm{FS}}(1+\sqrt{\brho})^2}\,,
\end{align*}
where $0<C_{\mathrm{FS}}<837$, as shown in Proposition \ref{prop:FS}. Furthermore
\[\mathds W_{\mathrm{FS}}^{-1}(\brho,u)=2\sqrt{\brho}\bigg\{\exp\Big(C_{\mathrm{FS}}^{2/3}\frac{(1+\sqrt{\brho})^{4/3}}{\brho^{2/3}}u^{2/3}\Big)-1 \bigg\}.\]
Theorem \ref{thm:key_eigen} applied here gives the following high probability bounds on RICs:
\begin{align*}
 \cmax & \leqslant \sqrt{\brho}(2+\sqrt{\brho})+t_0,\\
 \cmin & \leqslant \sqrt{\brho}(2-\sqrt{\brho})+t_0,
\end{align*}
where $t_0=\mathds W_{\mathrm{FS}}^{-1}(\brho,\frac{1}{\delta}\H(\brho\delta))$. 

Furthermore, Theorem \ref{thm:key_eigen} states that Condition \eqref{eq:Condition_SRSR} is satisfied with high probability whenever
\[ \delta > \frac{1}{2\rho}\exp\Big\{1-\frac{1}{2\rho}\mathds W_{\mathrm{FS}}\Big[2\rho,2\tau_0(\sqrt{\rho}-\sqrt{\rho_0})\big(\sqrt{\rho}-\frac{1}{2\sqrt{\rho_0}}\big)\Big] \Big\}.\]
When $\rho$ is small (which is the case in regimes \eqref{item:rho_0_delta_fixed} and \eqref{item:rho_0_delta_0}), the second argument in $\mathds W_{\mathrm{FS}}$ is approximately $\tau_0$ and this condition approximately writes
\[ \delta > \frac{1}{2\rho}\exp\Big[-\frac{|\log \rho|^{3/2}}{2^{3/2}C_{\mathrm{FS}}}\Big].\]

\begin{rem}
In the three asymptotic regimes \eqref{item:rho_0_delta_fixed}, \eqref{item:rho_fixed_delta_0} and \eqref{item:rho_0_delta_0}, these bounds behave badly compared to the ones by Bah and Tanner in \cite{bah2014asymptotics} (but note that we consider here entries which are not Gaussian anymore). This is indeed not surprising: from Section \ref{sec:ideal_dev}, it seems that the most important part in the rate function \eqref{eq:ideal_rate} is the moderate and large deviation behavior, whereas Feldheim and Sodin's inequality focuses only on small deviations. For matrices with independent Rademacher entries, concentration of measure phenomenon still occurs and one can use Talagrand's inequality to control large deviations of the largest singular value $\sigma_1$, which is a convex function of the entries. A similar inequality holds for matrices with independent sub-Gaussian entries:
\begin{equation}\label{eq:largest_sing_value_LPTJ}
\P\Big(\frac{\sigma_1(\X)}{\sqrt{n}}\geq t(1+\sqrt{\bar{\rho}})\Big) \leq e^{-cnt^2}, \quad \text{for all} \ t \geq C,
\end{equation}
where $\X$ is a $(r\times n)$ matrix with {\it iid} sub-Gaussian entries, with mean zero and variance $1$. Here $c,C>0$ depend on the sub-Gaussian moment of the entries and on $\bar{\rho}$ (see for example \cite{MR2146352}). Note that these constants might be computed explicitly but we did not pursue on this laborious task here. As far as we know, the more precise moderate and large deviation inequality for the smallest singular value of sub-Gaussian matrices was established by Rudelson and Vershynin (see for example \cite{RuVe_2010_non_asymptotic_extreme_singular_values}):
\begin{equation}\label{eq:smallest_sing_value_RV}
\P\Big(\frac{\sigma_r(\X)}{\sqrt{n}}\leq t(1-\sqrt{\bar{\rho}})\Big) \leq (Ct)^{n(1-\bar{\rho})}+e^{-cn}, \quad \text{for all} \ t \geq 0,
\end{equation}
where $c,C>0$ depend only on the sub-Gaussian moment of the entries. Again, these constants may be computed explicitly. Note that the $e^{-cn}$ term quantifies the fact that $\X$ is non singular with probability strictly less than $1$.

Combining Feldheim and Sodin's inequality with the preceding ones \eqref{eq:largest_sing_value_LPTJ} and \eqref{eq:smallest_sing_value_RV} would probably lead to a more accurate deviation inequality. However constants should be computed explicitly and the proof of Theorem \ref{thm:key_singular} modified in order to deal with the additive term $e^{-cn}$ in~\eqref{eq:smallest_sing_value_RV}. We will not pursue this task here.
\end{rem}

\subsection{Bounds on RICs and SRSR}
We summarize the bounds we obtained in the previous subsections. For sake of readability, we focus on the asymptotic Regime~\eqref{item:rho_0_delta_fixed}, in which $\brho \to 0$ and $\delta >0$ is fixed, so that the functions  $\Psi_{\min}^{(1)}$, $\Psi_{\max}^{(1)}$, $\Psi_{\min}^{(2)}$ and~$\Psi_{\max}^{(2)}$ have a simplest expression.

\medskip

\begin{center}
\begin{tabular}{| c || c | c |}
\hline
Inequality by & $\Psi_{\max}^{(1,2)}$  & $\Psi_{\min}^{(1,2)}$ \\
\hline
\hline
Davidson-Szarek & $2\sqrt{2}\sqrt{\brho|\log \brho|}+2\sqrt{\brho}+\sqrt{2}(1-\log \delta)\sqrt{\frac{\brho}{|\log \brho|}}+o\Big(\sqrt{\frac{\brho}{|\log \brho|}}\Big)$ & $= \Psi_{\max}^{(2)}$ \\
\hline
Ledoux-Rider & $C_{\mathrm{LR}}\sqrt{\brho}|\log \brho|+\Big[2+C_{\mathrm{LR}}\big(1-\log \delta\big)\Big]\sqrt{\brho}+o(\sqrt{\brho}) $ & $= \Psi_{\max}^{(1)}$ \\
\hline
Feldheim-Sodin & $2\sqrt{\brho}\exp\Big[C_{\mathrm{FS}}^{2/3}|\log \brho|^{2/3} \Big]+o\Big(\sqrt{\brho}\exp\Big[C_{\mathrm{FS}}^{2/3}|\log \brho|^{2/3} \Big]\Big)$ & $= \Psi_{\max}^{(1)}$  \\
\hline
\end{tabular}
\end{center}

\vspace{5mm}

We summarize next the conditions we obtained in the previous subsections on $\delta$ and $\rho$ so that Condition \eqref{eq:Condition_SRSR} is satisfied with high probability. For sake of readability again, this condition is written assuming that $\rho$ is small.

\vspace{5mm}

\begin{center}
\begin{tabular}{| c || c |}
\hline
Inequality by & Condition \eqref{eq:Condition_SRSR}\\
\hline
\hline
Davidson-Szarek & $ \delta > \frac{1}{2\rho}\exp\big[-\frac{\rho_0}{2\rho}\big]$ \\
\hline
Ledoux-Rider & $\delta > \frac{1}{2\rho}\exp\big[-\frac{\tau_0}{C_{\mathrm{LR}}\sqrt{2\rho}}\big]$ \\
\hline
Feldheim-Sodin  & $\delta > \frac{1}{2\rho}\exp\big[-\frac{|\log \rho|^{3/2}}{2^{3/2}C_{\mathrm{FS}}}\big]$ \\
\hline
\end{tabular}
\end{center}

\pagebreak[3]

\appendix 

\section{Proofs of the main results}

\subsection{Proof of Theorem~\ref{thm:Condition_SRSR}}
\label{app:proof_Condition_SRSR}

The proof follows the same guidelines as \cite[Proof of Theorem 6.13, page 145]{foucart2013mathematical}. A sufficient condition for SRSR is the $\ell_2$-robust null space property, see \cite[Theorem~4.22, page 88]{foucart2013mathematical}. Namely, we need to find constants $\kappa\in(0,1)$ and $\tau>0$ such that, for any $v\in\bbR^p$ and any $S\subset\{1,\ldots,p\}$ such that $|S|=s$,
\[
\|v_S\|_2\leq\frac\kappa{\sqrt s}\|v_{S^c}\|_1+\tau\|\M v\|_2\,.
\]
Given $v\in\bbR^p$, it is enough to consider $S=S_0$ the set of the $s$ largest (in magnitude) entries of $v$, $S_1$ the set of the $s$ largest (in magnitude) entries of $v$ in $S_0^c$, $S_2$ the set of the $s$ largest (in magnitude) entries of $v$ in $(S_0\cup S_1)^c$, etc. By definition of the RICs, one has
\begin{align*}
\|\M v_{S_0}\|_2^2=(1+&t)\|v_{S_0}\|_2^2\\
\mathrm{with}\quad-\cmin(2s,\M)\leq-\cmin(s,\M)\leq &t\leq\cmax(s,\M)\leq\cmax(2s,\M)\,.
\end{align*}
We begin with a first lemma. For sake of readability and from now on, $\cmin$ denotes $\cmin(2s,\M)$ and $\cmax$ denotes $\cmax(2s,\M)$.

\begin{lem}
\label{lem:p146}
For all $k\geq1$, it holds
\[
|\langle\M v_{S_0},\M v_{S_k}\rangle|\leq\sqrt{(\cmax-t)(\cmin+t)}\|v_{S_0}\|_2\|v_{S_k}\|_2\,.
\]
\end{lem}

\begin{proof}
Set $u=v_{S_0}/\|v_{S_0}\|_2$ and $w=\pm v_{S_k}/\|v_{S_k}\|_2$ where the sign of $w$ is chosen so that $|\langle\M u,\M w\rangle|=\langle\M u,\M w\rangle$. For $\alpha,\beta>0$ to be chosen later, it holds
\begin{align*}
2&|\langle\M u,\M w\rangle|=\frac1{\alpha+\beta}\Big[\|\M(\alpha u+w)\|_2^2-\|\M(\beta u-w)\|_2^2-(\alpha^2-\beta^2)\|\M u\|_2^2\Big]\\
&\leq \frac1{\alpha+\beta}\Big[(1+\cmax)\|\alpha u+w\|_2^2-(1-\cmin)\|\beta u-w\|_2^2-(\alpha^2-\beta^2)(1+t)\|u\|_2^2\Big]\\
&=\frac1{\alpha+\beta}\Big[(1+\cmax)(\alpha^2+1)-(1-\cmin)(\beta^2+1)-(\alpha^2-\beta^2)(1+t)\Big]\\
&=\frac1{\alpha+\beta}\Big[\alpha^2(\cmax-t)+\beta^2(\cmin+t)+\cmax+\cmin\Big]\,.
\end{align*}
Then, chose $\alpha=\sqrt{(\cmin+t)/(\cmax-t)}$ and $\beta=\sqrt{(\cmax-t)/(\cmin+t)}$ to get the desired inequality.
\end{proof}

\noindent
Using Lemma~\ref{lem:p146}, observe that
\begin{align*}
(1+t)\|v_{S_0}\|_2^2&=\|\M v_{S_0}\|_2^2\\
&=\langle\M v_{S_0},\M v\rangle-\sum_{k\geq1}\langle\M v_{S_0},\M v_{S_k}\rangle\\
&\leq \|\M v_{S_0}\|_2\|\M v\|_2+\sum_{k\geq1}\sqrt{(\cmax-t)(\cmin+t)}\|v_{S_0}\|_2\|v_{S_k}\|_2\\
&=\|v_{S_0}\|_2\Big[\sqrt{1+t}\|\M v\|_2+\sqrt{(\cmax-t)(\cmin+t)}\sum_{k\geq1}\|v_{S_k}\|_2\Big]\,.
\end{align*}
Now, Lemma 6.14 in \cite{foucart2013mathematical} gives that
\[
\sum_{k\geq1}\|v_{S_k}\|_2\leq\frac1{\sqrt s}\|v_{S_0^c}\|_1+\frac14\|v_{S_0}\|_2\,.
\]
We deduce that
\[
\|v_{S_0}\|_2\leq\frac b4\|v_{S_0}\|_2
			 +\frac{b}{\sqrt s}\|v_{S_0^c}\|_1
 			 +\frac{\|\M v\|_2}{\sqrt {1+t}}\,,
\]
where $b:={\sqrt{(\cmax-t)(\cmin+t)}}/{{(1+t)}}$. It follows that
\[
\|v_{S_0}\|_2\leq\frac{1}{\sqrt s}\frac{4b}{4-b}\|v_{S_0^c}\|_1
 			 +\frac{4}{4-b}\frac{\|\M v\|_2}{\sqrt {1+t}}\,.
\]
It suffices that $\kappa={4b}/({4-b})<1$ to get the $\ell_2$-robust null space property and hence SRSR. This is equivalent to $b={\sqrt{(\cmax-t)(\cmin+t)}}/{{(1+t)}}<4/5$. We have the following lemma.
\begin{lem}
For any $t\in[-\cmin,\cmax]$, it holds 
\[
\frac{\sqrt{(\cmax-t)(\cmin+t)}}{(1+t)}\leq\frac{\cmin+\cmax}{2\sqrt{(1-\cmin)(1+\cmax)}}\,.
\]
\end{lem}
\begin{proof}
Define $f(t)={{(\cmax-t)(\cmin+t)}}/{{(1+t)^2}}$ whose derivative is given by
\[
f'(t)=\frac{\cmax - \cmin - 2 \cmax \cmin - t (2 + \cmax-\cmin)}{(1+t)^3}
\]
We easily deduce that the function $f$ is upper bounded by the quantity $f(t^\star)$ where we denote $t^\star=(\cmax - \cmin - 2 \cmax \cmin)/(2 + \cmax-\cmin)$. Now, remark that it holds $f(t^\star)=(\cmin + \cmax)^2/(4 (1-\cmin)(1+\cmax))$. This gives the desired inequality.
\end{proof}

\noindent
It shows that SRSR holds whenever $({\cmin+\cmax})/{\sqrt{(1-\cmin)(1+\cmax)}}<8/5$. This last condition reads 
$\sqrt\gamma-1/\sqrt\gamma<8/5$ which is equivalent to $\sqrt\gamma<(4+\sqrt{41})/5$, where we denote $\gamma=\gamma(2s,\M)$. The desired condition follows.

\subsection{Proof of Theorem~\ref{thm:key_eigen}}
\label{sec:keylem}
We first prove \eqref{eq:bounds_RIC_eigen_max}. Set $\varepsilon>0$ and suppose that $\bar\rho$ and $\delta$ are such that $\frac{1}{\delta}\H(\bar\rho\delta)$ is in the range of $\mathds{W}(\bar\rho,\cdot)$. Set $t_0=\mathds{W}^{-1}\big(\bar\rho,\frac{1}{\delta}\H(\bar\rho\delta)\big)$. Then
\begin{align*}
\P\Big\{\cmax \geqslant \sqrt{\bar\rho}(2+\sqrt{\bar\rho})+t_0+\varepsilon\Big\} &=\P\Big\{1+\cmax \geqslant(1+\sqrt{\bar\rho})^2+t_0+\varepsilon\Big\}\\
& = \P\Big\{\exists x \in \Sigma_r \ \text{s.t.} \ \|\M x\|^2\geqslant \big((1+\sqrt{\bar\rho})^2+t_0+\varepsilon\big)\|x\|^2\Big\}\\
& \leqslant \sum_{\C_{r,n}} \P\Big\{\lambda_1(\C_{r,n}) \geqslant (1+\sqrt{\bar\rho})^2+t_0+\varepsilon\Big\}\\
&  \leqslant {p \choose r}c(\bar\rho)e^{-n\mathds W(\bar\rho,t_0+\varepsilon)}\\
   & \leqslant
c(\bar\rho)
   \Theta\,
e^{-n\mathds W(\bar\rho,t+\varepsilon)+p\H(r/p)}\\
   & \leqslant
   c(\bar\rho)
      \Theta\,
e^{-nD},
\end{align*}
with $\H(t)=-t\log t-(1-t)\log(1-t)$ for $t \in (0,1)$, $   \Theta^2:=e^{1/2}/(2\pi[r(1-r/p)]^{1/p})$ and $D=\mathds W(\bar\rho,t_0+\varepsilon)-\mathds{W}(\bar\rho,t_0)>0$. Indeed, note that Stirling formula (see Lemma~\ref{lem:AsymptoticGamma}) leads to
\[
 {p \choose r}\leq \frac{e^{1/4}}{\sqrt{2\pi}[r(1-r/p)]^{1/(2p)}}e^{-r\log(r/p)-(p-r)\log(1-r/p)}=\Theta e^{p\H(r/p)}\,.
\]
where $\Theta\to{e^{1/4}}/{\sqrt{2\pi}}$ when $0<r/p<1$ and $p$ goes to infinity. Set $D_2(\bar\rho,\delta,\varepsilon)=\frac{D}{2}$. Observe that $D-\log(\Theta)/n\geq D_2(\bar\rho,\delta,\varepsilon)$ for large enough $n$. Then, for~$n$ large enough (depending only on $\bar\rho$, $\delta$ and $\varepsilon$), it holds
\begin{align*}
 \P\Big\{\cmax \geqslant \sqrt{\bar\rho}(2+\sqrt{\bar\rho})+t_0+\varepsilon\Big\} 
& \leqslant c(\bar\rho)e^{-nD_2(\bar\rho,\delta,\varepsilon)}.
\end{align*}
Following the same arguments, we get a similar inequality \eqref{eq:bounds_RIC_eigen_min} for $\cmin$.

We suppose now that $(\M^{(n)})_{n\geq2}$ is a sequence of $(n\times p)$ matrices with $n/p\to\delta$ satisfying~\eqref{covariance_model} and that Condition~\eqref{eq:CondLemma} is satisfied, namely
\begin{equation*}
\delta>\Psi_0^{(1)}(\rho,\mathds W):=\frac{1}{2\rho}\exp\Big\{1-\frac{1}{2\rho}{\mathds W\Big[2\rho,2\tau_0(\sqrt{\rho}-\sqrt{\rho_0})\big(\sqrt\rho-\frac{1}{2\sqrt{\rho_0}}\big)\Big]}\Big\}\,.
\end{equation*}
Using the fact that $\H(t)\leq-t\log t+t$, this condition implies
\begin{equation}
\label{eq:t_*} 0< \frac{1}{\delta}\H(2\rho\delta)<\mathds{W}\Big[2\rho,2\tau_0(\sqrt{\rho}-\sqrt{\rho_0})\big(\sqrt{\rho}-\frac{1}{2\sqrt{\rho_0}}\big)\Big].
\end{equation}
Note that $2\tau_0(\sqrt{\rho}-\sqrt{\rho_0})(\sqrt{\rho}-\frac{1}{2\sqrt{\rho_0}})=\frac{8}{\sqrt{41}}\rho-2\sqrt{2}\sqrt{\rho}+\frac{4}{\sqrt{41}}$ belongs to $(0,\tau_0)$. Recall that $\mathds{W}(2\rho,.)$ is increasing on this interval. Applying $\mathds{W}^{-1}(2\rho,.)$ to \eqref{eq:t_*} leads to
\[ 0<t_0<\frac{8}{\sqrt{41}}\rho-2\sqrt{2}\sqrt{\rho}+\frac{4}{\sqrt{41}}.\]
Let $t_0<t<t_*=\frac{8}{\sqrt{41}}\rho-2\sqrt{2}\sqrt{\rho}+\frac{4}{\sqrt{41}}$. Remark that Condition~\eqref{eq:Condition_SRSR} is satisfied on the event $\{\cmax(2s,\M^{(n)}) < \sqrt{2\rho}(2+\sqrt{2\rho})+t\}\cap\{\cmin(2s,\M^{(n)}) < \sqrt{2\rho}(2-\sqrt{2\rho})+t\}$. Indeed, denoting $\gamma_0=\frac{(4+\sqrt{41})^2}{25}$,
\begin{align*}
\gamma(2s,\M)-\frac{(4+\sqrt{41})^2}{25}&=\frac{1+\cmax(2s,\M^{(n)})}{1-\cmin(2s,\M^{(n)})}-\gamma_0
\\&<\frac{(1+\sqrt{2\rho})^2+t}{(1-\sqrt{2\rho})^2-t}-\gamma_0
\\&<\frac{(1-\gamma_0)+2\sqrt{2}(1+\gamma_0)\sqrt{\rho}+2(1-\gamma_0)\rho+(1+\gamma_0)t}{(1-\sqrt{2\rho})^2-t}\,.
\end{align*}
Note that $t<t_*=\frac{8}{\sqrt{41}}\rho-2\sqrt{2}\sqrt{\rho}+\frac{4}{\sqrt{41}}$ implies that the denominator is positive.\\
Moreover $1-\gamma_0=-\frac{8}{25}(4+\sqrt{41})$ and $1+\gamma_0=\frac{2\sqrt{41}}{25}(4+\sqrt{41})>0$. Therefore
\begin{align*}
(1-\gamma_0)+2\sqrt{2}(1+\gamma_0)\sqrt{\rho}+2(1-\gamma_0)\rho+(1+\gamma_0)t &=(1+\gamma_0) (t-t_*) <0\,.
\end{align*}
As a consequence,
\[ \gamma(2s,\M)< \frac{(4+\sqrt{41})^2}{25}.\]

Invoking \eqref{eq:bounds_RIC_eigen_max} and \eqref{eq:bounds_RIC_eigen_min} with $\varepsilon=t-t_0$ and $\bar\rho=2\rho$ yields the following
\begin{align*}
  \P\Big(\cmax(2s,\M^{(n)}) \geqslant \sqrt{2\rho}(2+\sqrt{2\rho})+t\Big) & \leqslant c(2\rho)e^{-nD_2(2\rho,\delta,t-t_0)},\\
  \P\Big(\cmin(2s,\M^{(n)}) \geqslant \sqrt{2\rho}(2-\sqrt{2\rho})+t\Big) & \leqslant c(2\rho)e^{-nD_2(2\rho,\delta,t-t_0)},
\end{align*}
with $D_2(2\rho,\delta,t-t_0)=\big(\mathds{W}(2\rho,t)-\mathds{W}(2\rho,t_0)\big)/2>0$. Note that $t \mapsto D_2(2\rho,\delta,t-t_0)$ is continuous and increasing from $[t_0,t_*) \subset [0,\tau_0)$ onto $\big[0,D_2(2\rho,\delta,t_*-t_0)\big)$.\\
Choose $D_1(\rho,\delta) \in \big(0,D_2(2\rho,\delta,t_*-t_0)\big)$ (remark that $t_*$ and $t_0$ depend only on $\rho$ and $\delta$) and define $t_{\rho,\delta}$ in $(t_0,t_*)$ so that $D_2(2\rho,\delta,t_{\rho,\delta}-t_0)=D_1(\rho,\delta)$. Therefore
\begin{align*}
 & \P\Big\{{\M^{(n)}}\ \text{does not satisfy}\ \eqref{eq:Condition_SRSR}\ \text{with}\ s\leq\lfloor \rho n\rfloor\Big\}\\
 & \leqslant \P\Big(\cmax(2s,\M^{(n)}) \geqslant \sqrt{2\rho}(2+\sqrt{2\rho})+t_{\rho,\delta}\Big) + \P\Big(\cmin(2s,\M^{(n)}) \geqslant \sqrt{2\rho}(2-\sqrt{2\rho})+t_{\rho,\delta}\Big)\\
 & \leqslant 2c(2\rho)e^{-nD_1(\rho,\delta)},
\end{align*}
which concludes the proof.

 \subsection{Proof of Theorem~\ref{thm:key_singular}}
 \label{proof:key_singular}
 We follow the same lines as in the previous proof. We first prove~\eqref{eq:bound_RIC_singular_max}. For $\varepsilon>0$, we get the following.
\begin{align*}
\P\Big\{\cmax \geqslant (\sqrt{\bar\rho}+t_0)&(2+\sqrt{\bar\rho}+t_0) +\varepsilon\Big\}\\
& =\P\Big\{1+\cmax \geqslant(1+\sqrt{\bar\rho}+t_0)^2+\varepsilon\Big\}\\
&= \P\Big\{\exists x \in \Sigma_r \ \text{s.t.} \ \|\M x\|^2\geqslant \big((1+\sqrt{\bar\rho}+t_0)^2+\varepsilon\big)\|x\|^2\Big\}
 \\&\leqslant \sum_{\C_{r,n}} \P\Big\{\sigma_1 \geqslant \sqrt{(1+\sqrt{\bar\rho}+t_0)^2+\varepsilon}\Big\}\\
&\leqslant \sum_{\C_{r,n}} \P\Big\{\sigma_1 \geqslant 1+\sqrt{\bar\rho}+t_0+f(\bar\rho,\delta,\varepsilon)\Big\}\,,
\end{align*}
where $f(\bar\rho,\delta,\varepsilon)=\sqrt{(1+\sqrt{\bar\rho}+t_0)^2+\varepsilon}-( 1+\sqrt{\bar\rho}+t_0)>0$. Then
\begin{align*}
\P\Big(\cmax \geqslant (\sqrt{\bar\rho}+t_0)(2+\sqrt{\bar\rho}+t_0)+\varepsilon\Big)
 & 
 \leqslant {p \choose r}c(\bar\rho)e^{-n\mathds W\big(\bar\rho,t_0+f(\delta,\bar\rho,\varepsilon)\big)}\\
   & \leqslant
c(\bar\rho)
   \Theta\,
e^{-n\mathds W\big(\rho,t_0+f(\delta,\bar\rho,\varepsilon)\big)+p\H(r/p)}\\
   & \leqslant
   c(\bar\rho)
      \Theta\,
e^{-nD},
\end{align*}
with $\H(t)=-t\log t-(1-t)\log(1-t)$ for $t \in (0,1)$, $   \Theta^2:=e^{1/2}/(2\pi[r(1-r/p)]^{1/p})$ and $D=\mathds W\big(\bar\rho,t_0+f(\delta,\bar\rho,\varepsilon)\big)-\mathds{W}(\bar\rho,t_0)$. Set $D_2(\bar\rho,\delta,\varepsilon)=\frac{D}{2}$. Observe that $D-\log(\Theta)/n\geq D_2(\bar\rho,\delta,\varepsilon)$ for large enough $n$. Then, for~$n$ large enough (depending only on $\bar\rho$, $\delta$ and $\varepsilon$), it holds
\begin{align*}
 \P\Big\{\cmax \geqslant (\sqrt{\bar\rho}+t_0)(2+\sqrt{\bar\rho}+t_0)+\varepsilon\Big\} 
& \leqslant c(\bar\rho)e^{-nD_2(\bar\rho,\delta,\varepsilon)}.
\end{align*}
Following the same arguments, we get a similar inequality \eqref{eq:bound_RIC_singular_min} for $\cmin$.

We suppose now that $(\M^{(n)})_{n\geq2}$ is a sequence of $(n\times p)$ matrices with $n/p\to\delta$ satisfying~\eqref{singular_model} and that Condition~\eqref{eq:CondLemma2} is satisfied, namely
\begin{equation*}
\delta>\Psi_0^{(2)}(\rho,\mathds W):={\frac{1}{2\rho}}\exp\Big(1-\frac{1}{2\rho}{\mathds W\big[2\rho,\sqrt{2\rho_0}-\sqrt{2\rho}\big]}\Big)\,,
\end{equation*}
Using the fact that $\H(t)\leq-t\log t+t$, this condition implies
\[ 0< \frac{1}{\delta}\H(2\rho\delta)<\mathds{W}\Big(2\rho,\sqrt{2\rho_0}-\sqrt{2\rho}\Big).\]
Recall that $\mathds{W}(2\rho,.)$ is increasing on $[0,\sqrt{2\rho_0})$. Applying $\mathds{W}^{-1}(2\rho,.)$ leads to
\[ 0<t_0<\sqrt{2\rho_0}-\sqrt{2\rho}.\]
Let $t_0<t<t_*=\sqrt{2\rho_0}-\sqrt{2\rho}$. Remark that Condition~\eqref{eq:Condition_SRSR} is satisfied on the event $\{\cmax(2s,\M^{(n)}) < (\sqrt{2\rho}+t)(2+\sqrt{2\rho}+t)\}\cap\{\cmin(2s,\M^{(n)}) < (\sqrt{2\rho}+t)(2-\sqrt{2\rho}-t)\}$. Indeed
\begin{align*}
\sqrt{\gamma(2s,\M)}-\frac{4+\sqrt{41}}{5}&<\frac{1+\sqrt{2\rho}+t}{1-\sqrt{2\rho}-t}-\sqrt{\gamma_0}\\ 
& < \frac{1-\sqrt{\gamma_0}+\sqrt{2}(1+\sqrt{\gamma_0})\sqrt{\rho}+(1+\sqrt{\gamma_0})t}{1-\sqrt{2\rho}-t}\,.
\end{align*}
Note that $t<t_*=\sqrt{2\rho_0}-\sqrt{2\rho}$ implies that the denominator is positive. Moreover 
\begin{align*}
1-\sqrt{\gamma_0}+\sqrt{2}(1+\sqrt{\gamma_0})\sqrt{\rho}+(1+\sqrt{\gamma_0})t=(1+\sqrt{\gamma_0})(t-t_*)<0.
\end{align*}
As a consequence,
\[ \gamma(2s,\M)< \frac{(4+\sqrt{41})^2}{25}.\]

Similarly to \eqref{eq:bound_RIC_singular_max} and \eqref{eq:bound_RIC_singular_min}, it can be proved that for $t>t_0$,
\begin{align*}
  \P\Big(\cmax(2s,\M^{(n)}) \geqslant (\sqrt{2\rho}+t)(2+\sqrt{2\rho}+t)\Big) & \leqslant c(2\rho)e^{-nD_3(2\rho,\delta,t)},\\
  \P\Big(\cmin(2s,\M^{(n)}) \geqslant (\sqrt{2\rho}+t)(2-\sqrt{2\rho}-t)\Big) & \leqslant c(2\rho)e^{-nD_3(2\rho,\delta,t)},
\end{align*}
with $D_3(2\rho,\delta,t)>0$. Note that $t \mapsto D_3(2\rho,\delta,t)$ is continuous and increasing from $[t_0,t_*) \subset [0,\tau_0)$ onto $\big[D_3(2\rho,\delta,t_0),D_3(2\rho,\delta,t_*)\big)$. Choose $D_1(\rho,\delta) \in \big(D_2(2\rho,\delta,t_0);D_3(2\rho,\delta,t_*)\big)$ (note that $t_*$ and $t_0$ depend only on $\rho$ and $\delta$) and define $t_{\rho,\delta}$ so that $D_3(2\rho,\delta,t_{\rho,\delta})=D_1(\rho,\delta)$. Therefore
\begin{align*}
 \P\Big\{{\M^{(n)}}\ \text{does not satisfy}\ & \eqref{eq:Condition_SRSR}\ \text{with}\ s\leq\lfloor \rho n\rfloor\Big\}\\
 & \leqslant \P\Big(\cmax(2s,\M^{(n)}) \geqslant \big(\sqrt{2\rho}+t_{\rho,\delta}\big)\big(2+\sqrt{2\rho}+t_{\rho,\delta}\big)\Big)\\
 & \qquad + \P\Big(\cmin(2s,\M^{(n)}) \geqslant \big(\sqrt{2\rho}+t_{\rho,\delta}\big)\big(2-\sqrt{2\rho}-t_{\rho,\delta}\big)\Big)\\
 & \leqslant 2c(2\rho)e^{-nD_1(\rho,\delta)},
\end{align*}
which concludes the proof.


 \bibliographystyle{abbrv}
 \bibliography{biblio}

\newpage

\section{Supplement : Deviations for the Rademacher model} \label{sec:dev_Rad}
In this section we follow the steps of the work \cite{feldheim2010universality} to get small deviation inequalities on the extreme eigenvalues of Gram matrices built from the Rademacher law. The paper \cite{feldheim2010universality} focuses on the asymptotic distribution of the fluctuations of the extreme eigenvalues, and it proved that the extreme eigenvalues of the sample covariance matrices built from sub-Gaussian matrices asymptotically fluctuate around their limiting values (with proper scaling) with respect to the Tracy-Widom distribution. Their results follow from an interesting estimation of the moments of the fluctuations. While their estimation is interestingly of the right order (namely $\varepsilon^{3/2}$), the authors of \cite{feldheim2010universality} did not pursue on giving an upper bound of the constant appearing in their rate function, see Claim $(a)$ and $(b)$ of Point $2$ in \cite[Corollary V.2.1]{feldheim2010universality}. 

Unfortunately, the constant $C_{\mathrm{FS}}$ appearing in the rate function is of crucial importance when deriving phase transitions, see Section \ref{sec:main} for instance. Hence, we need to track the proof of \cite{feldheim2010universality} in order to provide an upper bound on $C_{\mathrm{FS}}$ and its dependence on the ratio $\rho$ of the sizes of the Rademacher matrix. This strenuous hunt necessitates to recast all the asymptotic bounds appearing in \cite{feldheim2010universality} into non asymptotic ones as sharp as possible. The benefit of this elementary but non trivial task is the following. It gives, for the first time, an explicit expression of small deviations of extreme eigenvalues of the sample covariance matrices at the sharp rate $\varepsilon^{3/2}$. This section is devoted to prove the following result.

\begin{prop}
\label{prop:FS}
Let $N>M\geq 54$ and consider 
\[
\C:=\X\X^\top\quad\mathrm{where}\ \X\in\{\pm1\}^{M\times N}\mathrm{\ with\ i.i.d.\ Rademacher\ entries}
\] 
 then 
 \begin{align*}
\Prob\Big\{\lambda_M(\B)  \geqslant (\sqrt{M}+\sqrt{N})^2+\varepsilon N\Big\}
&\leq
\frac{\mathds W_0(\rho,\varepsilon)}
{{1-\rho}}
M
\exp(-N\mathds W_{\mathrm{FS}}(\rho,\varepsilon))
\\
\Prob\Big\{\lambda_1(\B)  \leqslant (\sqrt{M}-\sqrt{N})^2-\varepsilon N\Big\}
&\leq
\frac{\mathds W_0(\rho,\varepsilon)}
{{1-\rho}}
M
\exp(-N\mathds W_{\mathrm{FS}}(\rho,\varepsilon))
 \end{align*}
where $\rho=M/N$ and 
\begin{align*}
\mathds W_0(\rho,\varepsilon)&:=
c_0\exp
\left[
{c_0\sqrt{\log\Big(1+\frac{\varepsilon}{2\sqrt{\rho}}\Big)}}
\right]
\\
\mathds W_{\mathrm{FS}}(\rho,\varepsilon)&:=
\frac{\rho\log(1+\frac{\varepsilon}{2\sqrt{\rho}})^{\frac32}}{C_{\mathrm{FS}}(1+\sqrt{\rho})^2}
\end{align*}
for some universal constants $c_0>0$ and $837>C_{\mathrm{FS}}>0$.
Furthermore, for any $C>3242$, there exists a constant $v:=v(\rho,C)>0$ that depends only on $\rho=M/N$ and $C$ such that, for all $0<\varepsilon<\sqrt{\rho}$,
 \begin{align*}
\Prob\Big\{\lambda_M(\B)  \geqslant (\sqrt{M}+\sqrt{N})^2+\varepsilon N\Big\}
&\leq
v
\exp\Big(-C^{-1}N\frac{\rho^{1/4}}{(1+\sqrt{\rho})^2}\varepsilon^{\frac32}\Big)
\\
\Prob\Big\{\lambda_1(\B)  \leqslant (\sqrt{M}-\sqrt{N})^2-\varepsilon N\Big\}
&\leq
v
\exp\Big(-C^{-1}N\frac{\rho^{1/4}}{(1+\sqrt{\rho})^2}\varepsilon^{\frac32}\Big)\,.
\end{align*}

\end{prop}

\subsection{Sketch of the proof}
The result of \cite{feldheim2010universality} is based on a combinatorial proof. Interestingly, this approach is suited for the Rademacher model since, in this case, traces of polynomials of the covariance matrix $\C$ can be expressed as the number of non-backtracking paths of given length. In this section, we change notation and we use the notation of the paper \cite{feldheim2010universality} to ease readability when referring to this latter. Hence, we consider a Rademacher matrix of size $(M\times N)$ with $M<N$ (referred to as $(s\times n)$ with $s<n$ in the rest of this paper). We draw this proof into the following points.
\begin{enumerate}
\item 
The proof \cite{feldheim2010universality} is based on a moment method that captures the influence of the largest and the smallest eigenvalues considering a new centering
\[
\tB:=\frac{\C-(M+N-2)}{2\sqrt{(M-1)(N-1)}}\,.
\]
 The authors \cite{feldheim2010universality} then use the trace of $\tB^{2m}+\tB^{2m-1}$ (resp. $\tB^{2m}-\tB^{2m-1}$) to estimate the moments of the largest (resp. smallest) eigenvalue. 
\item The control of 
\[
A_m:=\Exp\Tr{\tB^{2m}}+\Exp\Tr{\tB^{2m-1}}
\]
(resp. $B_m:=\Exp\Tr{\tB^{2m}}-\Exp\Tr{\tB^{2m-1}}$) is given by a control of traces of polynomials $Q_n(\C)$ of $\C$. Up to a proper scaling, these polynomials are the orthogonal polynomials of the Marchenko-Pastur law which can be expressed by Chebyshev polynomials $U_n$ of the second kind.
\item 
In the Rademacher model, the aforementioned traces, namely $\Exp\Tr{Q_n(\C)}$, are exactly the number $\hat\Sigma_{1}^{1}(n)$ of non-backtracking paths on the complete bi-partite graph that cross an even number of times each edge and end at their starting vertex. This claim can be generalized to general random sub-Gaussian matrices, up to technicalities.
\item 
To estimate the number of non-backtracking paths $\hat\Sigma_{1}^{1}(n)$, the article \cite{feldheim2010universality} begins with a mapping from the collection of non-backtracking paths into the collection of weighted diagrams. Then it provides an automaton which constructs all possible diagrams. The number of diagrams constructed by the automaton ending in $s$ steps is denoted $D_1(s)$. Lemma~\ref{lem:D1} provides an upper bound on this quantity.  Summing over $s$, it yields an upper bound on $\hat\Sigma_{1}^{1}(n)$, see \eqref{eq:PathsBound} in Lemma \ref{lem:PathsBound}.
\item In the Rademacher model, $\hat\Sigma_{1}^{1}(n)$ is the expectation of the trace of $Q_{n}$. Hence, we deduce an upper bound on these traces.
\item Using Markov inequality and optimizing over the parameters, we deduce small deviation inequalities on the smallest and largest eigenvalues.
\end{enumerate}

\subsection{Number of diagrams}
Recall that $D_1(s)$ denotes the number of diagrams constructed by the automaton ending in $s$ steps. The description of the automaton can be found in \cite{feldheim2010universality} Section II.2 page 101.
\begin{lem}
\label{lem:D1}
It holds, for all $s\geq1$, 
\[
D_{1}(s)\leqslant C_{0,D}C_{D}^{s-1}s^{s-1/2}
\] 
where $C_{0,D}$ and $C_{D}$ can be chosen as $C_{0,D}=8.31$ and $C_{D}=53.8$.
\end{lem}
\begin{proof}
We follow Proposition II.2.3 of \cite{feldheim2010universality} but we focus on the case (of sample covariance matrices) corresponding to $\beta=1$. In this case, there are three types of transitions from one state to the following one. Let $s=2g+h$ be the number of steps in the automaton at the end, where $h$ is the number of transition of type $3$ and $g$ the number of transition of type $1$.

\noindent
$\bullet$ If $h=0$ then the number of ways to order the transitions of the type $1$ and $2$ is exactly $\frac{(2g)!}{g!(g+1)!}$.  Informally, the state of the automaton can be seen as a ``thread'' made of straight pieces and loops. The total length of this thread changes at each step. These changes of length are encoded by non-negative integers $m_i$. For precise definition of these numbers, see \cite{feldheim2010universality} Section II.2 page 103. In the present case, the number of ways to choose the numbers $m_{i}$ is at most ${6g-1 \choose 4g}$. The number of diagrams corresponding to a fixed order of transitions and fixed $m_{i}$ is at most $(6g-1)^{2g}$ (indeed, the following state is then determined by choosing an edge and there are $6g-1$ edges in the diagram). As in \cite{feldheim2010universality}, we deduce that an upper on $D_{1}$ is 
\[
\frac{(2g)!}{g!(g+1)!}{6g-1 \choose 4g}(6g-1)^{2g}=\frac{2g(6g-1)!(6g-1)^{2g}}{g!(g+1)!(4g)!}\,.
\]
Using Lemma \ref{lem:AsymptoticGamma}, this number is upper bounded by
\[
\frac {e^{2+1/60}}{\pi}g\frac{(6g-1)^{8g-1/2}}{g^{g+1/2}(g+1)^{g+3/2}(4g)^{4g+1/2}}\, .
\]
Writing $\theta=\frac{(6g-1)^{8g-1/2}}{g^{g+1/2}(g+1)^{g+3/2}(4g)^{4g+1/2}}$ in exponential form, we get
\[ \theta=\exp\Big(2g\log g+g\big(8\log(6)-4\log(4)\big)-3\log g-\frac{1}{2}(3\log 2+\log 3)+\gamma(g)\Big),
\]
with
\[ \gamma(g)=\Big(8g-\frac{1}{2}\Big)\log\Big(1-\frac{1}{6g}\Big)-\Big(g+\frac{3}{2}\Big)\log\Big(1+\frac{1}{g}\Big).
\]
Note that $\gamma$ is non decreasing on $(1,\infty)$ and goes to $-\frac{7}{3}$ when $g \to \infty$. Therefore, the number of diagrams in this case is upper bounded by (recall that $s=2g$ here)
\[\frac{1}{\pi}e^{7/2\log(3)-3/2\log(2)+1/60-1/3}(40.5)^{s-1}s^{s-2} \leqslant 3.84(40.5)^{s-1}s^{s-2}.\]

\noindent
$\bullet$ If $g=0$ then there are only transitions of the third kind. The number of ways to choose the numbers $m_{i}$ is at most ${2h-1 \choose h{-1}}
$. The number of diagrams corresponding to a fixed order of transitions and fixed $m_{i}$ is at most $(3h-1)^{h}$ (indeed, recall that the number of edges of the diagram is $3h-1$). We deduce that an upper bound on $D_{1}$ is 
\[
\frac{(2h-1)!}{h!(h-1)!}(3h-1)^{h}\,.
\]
Note that this number is $2$ when $h=1$. For $h\geqslant 2$, using Lemma \ref{lem:AsymptoticGamma}, this number is upper bounded by
\[
\frac {e^{1/12}}{\sqrt{2\pi}}\frac{(2h-1)^{2h-1/2}(3h-1)^h}{h^{h+1/2}(h-1)^{h-1/2}}\, .
\]
Once again, we write $\theta=\frac{(2h-1)^{2h-1/2}(3h-1)^h}{h^{h+1/2}(h-1)^{h-1/2}}$ in exponential form. This yields
\[\theta=\exp\Big[\Big(h-\frac{1}{2}\Big)\log h+\big(2\log(2)+\log(3)\big)(h-1)+\frac{3}{2}\log2+\log(3)+\gamma(h) \Big]\, ,\]
with \[\gamma(h)=\Big(2h-\frac{1}{2}\Big)\log\Big(1-\frac{1}{2h}\Big)+h\log\Big(1-\frac{1}{3h}\Big)-\Big(h-\frac{1}{2}\Big)\log\Big(1-\frac{1}{h}\Big).\]
Note that $\gamma$ is non increasing on $(2,h^*)$ and non decreasing on $(h^*,\infty)$ for some $h^*>2$. Therefore, $\gamma(h)$ is bounded by $\max(\gamma(2),\lim_{h \to \infty}\gamma(h))$. This yields $\gamma(h) \leqslant -0.33$ for all $h \geqslant 2$. Finally, the number of diagrams in this case is upper bounded by (recall that $s=h$ here)
\[\frac{e^{1/12}}{\sqrt{2\pi}}e^{3/2\log(2)+\log(3)-0.33}(12)^{s-1}s^{s-1/2} \leqslant 2.65(12)^{s-1}s^{s-1/2}.\]

\noindent
$\bullet$ If $h\neq 0$ and $g\neq0$ then the number of ways to order the transitions of the three types is exactly 
\[
{2g+h \choose h}\frac{(2g)!}{g!(g+1)!}=\frac{(2g+h)!}{h!g!(g+1)!}\,.
\]
The number of ways to choose the numbers $m_{i}$ is at most ${6g+2h-1 \choose 2g+h{-1}}$. The number of diagrams corresponding to a fixed order of transitions and fixed $m_{i}$ is at most $(6g+3h-1)^{2g+h}$ (indeed, recall that the number of edges of the diagram is $6g+3h-1$). 

We deduce that an upper bound on $D_{1}$ is 
\[\frac{(2g+h)!}{h!g!(g+1)!}{6g+2h-1 \choose 2g+h{-1}}(6g+3h-1)^{2g+h}.\]
Using the fact that $s=2g+h$ and Lemma \ref{lem:AsymptoticGamma}, this number is bounded by
\[
\frac{e^{131/126}}{(2\pi)^{3/2}}\frac{s^{s+1/2}(3s-h-1)^{3s-h-1/2}(3s-1)^s}{h^{h+1/2}g^{g+1/2}(g+1)^{g+3/2}(s-1)^{s-1/2}(2s-h)^{2s-h+1/2}}.
\]
Let $t=h/s\in[1/s,1-2/s]$ so that an upper bound is
\[
\frac{e^{131/126}}{(2\pi)^{3/2}}\frac{s^{s+1/2}(3s-ts-1)^{3s-ts-1/2}(3s-1)^s}{(ts)^{ts+1/2}(s\frac{1-t}{2})^{s(1-t)/2+1/2}(s\frac{1-t}{2}+1)^{s(1-t)/2+3/2}(s-1)^{s-1/2}(2s-ts)^{2s-ts+1/2}}.
\]
Once again, we write this in exponential form and get
\[
\frac{e^{131/126}}{(2\pi)^{3/2}}\exp\Big(s\log s-\frac{5}{2}\log s+\beta(t)s+\alpha(t)+\gamma(s,t)\Big),
\]
with
\begin{align*}
\alpha(t)&=2\log 2 -\frac{1}{2}\log(3-t)-\frac{1}{2}\log t-2\log(1-t)-\frac{1}{2}\log(2-t),\\
\beta(t)&=(3-t)\log(3-t)+\log(3)\\
&\quad-t\log t-(1-t)\log(1-t)+\log(2)(1-t)-(2-t)\log(2-t),\\
\gamma(s,t)&=\Big((3-t)s-\frac{1}{2}\Big)\log\Big(1-\frac{1}{(3-t)s}\Big)+s\log\Big(1-\frac{1}{3s}\Big)\\
&\quad-\frac{1}{2}\Big(s(1-t)+3\Big)\log\Big(1+\frac{2}{s(1-t)}\Big)-\Big(s-\frac{1}{2}\Big)\log\Big(1-\frac{1}{s}\Big).
\end{align*}

$\circ$
We focus first on $\beta$. This function is non decreasing on $(0,t^*)$ and non increasing on $(t^*,1)$, with $t^*=\frac{3}{2}-\frac{\sqrt{57}}{6}\approx 0.24$. Therefore, it reaches its maximum at $t^*$. Computing it yields $\beta(t) \leqslant 3.985$ for all $t \in (0,1)$.

$\circ$
We focus now on $\alpha$. This function is non increasing on $(0,t')$ and non decreasing on $(t',1)$ with $t' \in (0,1)$. Recall that $t \in (1/s,1-2/s)$. Therefore, $\alpha(t) \leqslant \max(\alpha(1/s),\alpha(1-2/s))$. Computing these two values and using the fact that $s \geqslant 3$ leads to $\alpha(t) \leqslant \alpha(1-2/s)$ for all $t \in (1/s,1-2/s)$. Consequently
\[\alpha(t) \leqslant 2\log s -\frac{1}{2}\log\Big(2+\frac{2}{s}\Big)-\frac{1}{2}\log\Big(1-\frac{2}{s}\Big)-\frac{1}{2}\log\Big(1+\frac{2}{s}\Big).\]

$\circ$
Let's turn to $\gamma$. Recall that $t \in (1/s,1-2/s)$. Dealing separately with the two terms $\big((3-t)s-\frac{1}{2}\big)\log\big(1-\frac{1}{(3-t)s}\big)$ and $\frac{1}{2}(s(1-t)+3)\log(1+\frac{2}{s(1-t)})$ yields
\begin{align*}
 \gamma & \leqslant \Big(3s-\frac32\Big)\log\Big(1-\frac{1}{3s-1}\Big)+s\log\Big(1-\frac{1}{3s}\Big)\\
 & \quad -\frac{1}{2}(s+2)\log\Big(1+\frac{2}{s-1}\Big)-\Big(s-\frac{1}{2}\Big)\log\Big(1-\frac{1}{s}\Big).
\end{align*}

\noindent
Going back to the number of diagrams in this case, it is bounded by
\[
\frac{e^{131/126}}{(2\pi)^{3/2}}\exp\Big(s\log s-\frac{1}{2}\log s+3.985s+\delta(s)\Big),
\]
with
\begin{align*}
\delta(s)  =&-\frac{1}{2}\log\Big(2+\frac{2}{s}\Big)-\frac{1}{2}\log\Big(1-\frac{2}{s}\Big)-\frac{1}{2}\log\Big(1+\frac{2}{s}\Big)\\
 &+\Big(3s-\frac32\Big)\log\Big(1-\frac{1}{3s-1}\Big)+s\log\Big(1-\frac{1}{3s}\Big)\\
 &-\frac{1}{2}(s+2)\log\Big(1+\frac{2}{s-1}\Big)-\Big(s-\frac{1}{2}\Big)\log\Big(1-\frac{1}{s}\Big).
 \end{align*}
This function is non decreasing on $(3,\infty)$ and goes to $-\frac{4}{3}-\frac{\log2}{2}\leqslant -1.67$ when $s$ goes to $\infty$. Therefore, there are at most
\[
\frac{e^{131/126-1.67}}{(2\pi)^{3/2}}(e^{3.985})^{s-1}s^{s-1/2} \leqslant 1.82 (53.8)^{s-1}s^{s-1/2}
\]
diagrams in this case.
\noindent
This leads to the result.
\end{proof}

\subsection{Number of paths}
Let $n\geq1$ be fixed. Recall that $\Exp\Tr{Q_n(\B)}$ is equal to the number $\hat\Sigma_{1}^{1}(n)$ of non-backtracking paths, see page 115 in \cite{feldheim2010universality}. Recall that $M\leqslant N$ denotes the sizes of the {Rademacher} matrix.
\begin{lem}
\label{lem:PathsBound}
It holds
\eq
\label{eq:PathsBound}
\hat\Sigma_{1}^{1}(n)\leqslant C_{0,\hat\Sigma}n(MN)^{n/2}\exp\Big[\frac{C_{\hat\Sigma}(1+\sqrt{M/N})n^{3/2}}{\sqrt {M}}\Big]
\qe
where $C_{0,\hat\Sigma}=160.4$  and $C_{\hat\Sigma}=13.3$. As a consequence,
\[\Exp[\Tr Q_n(\B)] \leqslant C_{0,\hat\Sigma}(MN)^{n/2}n\exp\Big(C_{\hat\Sigma}(1+\sqrt{M/N})\frac{n^{3/2}}{M^{1/2}}\Big).\]
\end{lem}

\begin{proof}
The number of diagrams is $D_{1}(s)$ for $1\leqslant s\leqslant n$. The number of ways to choose the vertices on a diagram constructed in $s$ steps by the automaton is at most
\[
\frac12(MN)^{n/2}\Big[(1+\sqrt{M/N})(M^{-1/2}+N^{-1/2})^{2s-2}+(1-\sqrt{M/N})(M^{-1/2}-N^{-1/2})^{2s-2}\Big]\,,
\]
see \cite[page 117]{feldheim2010universality}. The number of ways to choose the weights on a diagram constructed in $s$ steps by the automaton is at most \cite{feldheim2010universality}
\[
\frac{(3s+1)}{(3s-2)!}\Big(\frac{n-3s+1}2+3s-2\Big)^{3s-2}\,.
\]
\noindent
We deduce that the number $\hat\Sigma_{1}^{1}(n)$ of non-backtracking paths is at most
\[
\hat\Sigma_{1}^{1}(n)\leq\frac12(MN)^{n/2}\bigg[\Big(1+\sqrt{\frac MN}\Big)T_{1}+\Big(1-\sqrt{\frac MN}\Big)T_{2}\bigg]
\]
where
\begin{align*}
T_{1}&:=\sum_{s=1}^{n}D_{1}(s)(M^{-1/2}+N^{-1/2})^{2s-2}\frac{(3s+1)}{(3s-2)!}\Big(\frac{n-3s+1}2+3s-2\Big)^{3s-2}
\\
T_{2}&:=\sum_{s=1}^{n}D_{1}(s)(M^{-1/2}-N^{-1/2})^{2s-2}\frac{(3s+1)}{(3s-2)!}\Big(\frac{n-3s+1}2+3s-2\Big)^{3s-2}
\end{align*}
We can bound each term. It reads as follows.
\begin{align*}
T_{1}&\leqslant C_{0,D}\sum_{s=1}^{n}C_{D}^{s-1}s^{s-1/2}(M^{-1/2}+N^{-1/2})^{2s-2}\frac{(3s+1)}{(3s-2)!}\Big(\frac{n-3s+1}2+3s-2\Big)^{3s-2}
\\
&\leqslant  C_{0,D}\sum_{s=1}^{n}C_{D}^{s-1}\Big[\frac{1+\sqrt{M/N}}{\sqrt M}\Big]^{2(s-1)}\frac{(3s+1)(n+3s-3)^{3s-2}s^{s-1/2}}{(3s-2)!2^{3s-2}}
\end{align*}
using Lemma \ref{lem:D1}. Invoke Lemma \ref{lem:AsymptoticGamma} to get that
\begin{align*}
\frac{(2(s-1))!}{n^{3(s-1)}} & \frac{(3s+1)({n+3s-3)^{3s-2}}s^{s-1/2}}{(3s-2)!2^{3s-2}}\\
&\leqslant n\frac{e^{s+1/12}}{2^{3s-2}}(3s+1)\sqrt{\frac{2s-2}{3s-2}}\Big(1+\frac{3s-3}n\Big)^{3s-2}\frac{(2s-2)^{2s-2}s^{s-1/2}}{(3s-2)^{3s-2}}
\\
&\leqslant ne^{1/12}\sqrt{\frac{2s-2}{3s-2}}2^{3s-2}(3s+1)\frac{(2s-2)^{2s-2}s^{s-1/2}}{(3s-2)^{3s-2}}e^s.
\end{align*}
But $2^{3s-2}(3s+1)\frac{(2s-2)^{2s-2}s^{s-1/2}}{(3s-2)^{3s-2}}e^s\leqslant \exp(s+f(s))$ where \[f(s)=(3s-2)\log(2)+(2s-2)\log(2s-2)+\log(3s+1)+\Big(s-\frac{1}{2}\Big)\log(s)-(3s-2)\log(3s-2).\]
Some elementary computations give the following:
\begin{align*}
 f(s)& =\frac{1}{2}\log s +(5\log 2-3\log 3)s+3\log 3-4\log 2+(2s-2)\log\Big(1-\frac{1}{s}\Big)\\
 & \qquad -(3s-2)\log\Big(1-\frac{2}{3s}\Big)+\log\Big(1+\frac{1}{3s}\Big)\\
 & = (5\log 2-3\log 3)s+3\log 3-4\log 2+g(s),
\end{align*}
with $g(s)=\frac{1}{2}\log s+(2s-2)\log\big(1-\frac{1}{s}\big)-(3s-2)\log\big(1-\frac{2}{3s}\big)+\log\big(1+\frac{1}{3s}\big)$.
We have
\begin{align*}
g'(s)&=\frac{3s-1}{2s(3s+1)}+2\log\Big(1-\frac{1}{s}\Big)-3\log\Big(1-\frac{2}{3s}\Big),\\
g''(s)&=\frac{-27s^4+99s^3-21s^2+11s+2}{2s^2(s-1)(3s-2)(3s+1)^2}.
\end{align*}
It may be shown that there exists $s_*>2$ such that $g''$ is positive on $(1,s_*)$ and negative on $(s_*,\infty)$. Therefore, $g$ is strictly concave on $(s_*,\infty)$ and its curve is below its tangents, which write $y=g'(s_0)(s-s_0)+f(s_0)$. For $s \in [1,s_*]$, $g(s) \leqslant g(1)=2\log 2$. As a consequence, we are looking for the point $s_0 \in (s_*,\infty)$ such that the tangent at $s_0$ goes through the point $(1,2\log 2)$. This tangent goes through the point $(1,g(s_0)+(1-s_0)g'(s_0))$. Set $h(s)=g(s)+(1-s)g'(s)$. This function is non decreasing and there is a unique point $s_0 \in (s_*,\infty)$ such that $h(s_0)=2\log 2$. It may be shown that $s_0 \in (39.66;39.67)$. As $g'$ is non increasing on this interval, $g'(s_0) \leqslant g'(39.66)\leqslant 0.013$. This leads to
\[ g(s) \leqslant 0.013(s-1)+2\log 2.\]
Then
\begin{align*}
\frac{(2(s-1))!}{n^{3(s-1)}} & \frac{(3s+1)({n+3s-3)^{3s-2}}s^{s-1/2}}{(3s-2)!2^{3s-2}}\\
& \leqslant ne^{1/12}\sqrt{\frac{2}{3}}\exp\Big((5\log 2-3\log 3+1.013)s+3\log 3-2\log 2-0.013\Big)\\
&\leqslant ne^{1/12}\sqrt{\frac{2}{3}}\exp(1+3\log 2)\exp\Big((5\log 2-3\log 3+1.013)(s-1)\Big)\\
& \leqslant 19.3 \ n \ (3.27)^{s-1}.
\end{align*}
As a consequence,
\begin{align*}
 T_1 & \leqslant 19.3 C_{0,D}n\sum_{s=1}^{n}\frac{1}{(2(s-1))!}\Big[\frac{1.81(1+\sqrt{M/N})n^{3/2}}{\sqrt {C_{D}^{-1}M}}\Big]^{2(s-1)}\\
 & \leqslant 19.3 C_{0,D}n\exp\Big(1.81\sqrt {C_{D}}(1+\sqrt{M/N})\frac{n^{3/2}}{M^{1/2}}\Big).
\end{align*}
Similarly, one gets
\begin{align*}
T_{2}&\leqslant 19.3n C_{0,D}\sum_{s=1}^{n}\frac{1}{(2(s-1))!}\Big[\frac{1.81(1-\sqrt{M/N})n^{3/2}}{\sqrt {C_{D}^{-1}M}}\Big]^{2(s-1)}\\
&\leqslant 19.3n C_{0,D}\exp\Big[\frac{1.81(1-\sqrt{M/N})\sqrt{C_{D}}n^{3/2}}{\sqrt {M}}\Big]\,.
\end{align*}
This yields the result.
\end{proof}

\subsection{Bound on the traces}\label{sec:trace}
\begin{lem}
\label{lem:traces}
It holds that
\eq
\label{eq:TraceBound}
\Big(\Exp[\Tr{\tB^{2m}}]+\Exp[\Tr{\tB^{2m-1}}]\Big)\vee \Big(\Exp[\Tr{\tB^{2m}}]-\Exp[\Tr{\tB^{2m-1}}]\Big)\leq\Delta_m
\qe
where 
\begin{align*}
\Delta_m&=
\frac{C_{0,\mathrm{Rad}}}{1-\frac{M}{N}}
m
 \Big[{\Big(\frac{MN}{(M-1)(N-1)}\Big)^m}+\frac Mm\Big]
\exp\Big( C_{\mathrm{Rad}}(1+\sqrt{M/N})^4\frac{m^3}{M^2}\Big),
\end{align*}
and
\begin{align*}
C_{0,\mathrm{Rad}}&=594C_{0,\hat\Sigma}=95,278\\
C_{\mathrm{Rad}}&=355.7C_{D}^2=830,415.
\end{align*}
\end{lem}

\begin{proof}
Invoke Lemma IV.1.1 page 115 in \cite{feldheim2010universality} and Lemma \ref{lem:PathsBound} to get that
\begin{align}
\label{eq:TraceVks}
 \Exp[\Tr {V_{n,\frac{(M-2)^2}{(M-1)(N-1)}}(\tB)}] & \leqslant C_{0,\hat\Sigma}\textcolor{red}{}{\Big(\frac{MN}{(M-1)(N-1)}\Big)^{n/2}}n\exp\Big(C_{\hat\Sigma}(1+\sqrt{M/N})\frac{n^{3/2}}{M^{1/2}}\Big)\,.
 \end{align}
Set $s:=\frac{(M-2)^2}{(M-1)(N-1)}$. For $m \geqslant 1$, let $A_m=\Exp[\Tr{\tB^{2m}}]+\Exp[\Tr{\tB^{2m-1}}]$. Following pages 95-96 in \cite{feldheim2010universality} yields:
\begin{align}
 A_m & =\frac{1}{(2m+1)2^{2m}}\sum_{n=0}^m(2n+1)\binom{2m+1}{m-n}\Exp[\Tr{U_{2n}(\tB)}]\notag\\
 & \quad +\frac{1}{2m2^{2m}}\sum_{n=1}^m2n\binom{2m}{m-n}\Exp[\Tr{U_{2n-1}(\tB)}].\label{eq:Am1}
\end{align}
Using the fact that $V_{k,s}=U_k+\sqrt{s}U_{k-1}$, it holds
\begin{align}
 A_m & =\frac{1}{(2m+1)2^{2m}}\sum_{n=0}^m(2n+1)\binom{2m+1}{m-n}\sum_{k=0}^{2n}(-1)^ks^{k/2}\Exp[\Tr{V_{2n-k,s}(\tB)}]\notag\\
 & \quad +\frac{1}{2m2^{2m}}\sum_{n=1}^m2n\binom{2m}{m-n}\sum_{k=0}^{2n-1}(-1)^ks^{k/2}\Exp[\Tr{V_{2n-k-1,s}(\tB)}].
\notag
\end{align}
Note that the expectation $\Exp[\Tr {V_{k,s}(\tB)}]$ is non-negative. Indeed, one can check that, up to a multiplicative constant, $\Exp[\Tr {V_{k,s}(\tB)}]=\Exp[\Tr {Q_{k}(\tB)}]=\hat\Sigma_{1}^{1}(k)$. It follows that
\begin{align}
 A_m & \leqslant \frac{1}{(2m+1)2^{2m}}\sum_{n=0}^m(2n+1)\binom{2m+1}{m-n}\sum_{k=0}^{n}s^{k}\Exp[\Tr{ V_{2(n-k),s}(\tB)}]\notag\\
 & \quad +\frac{1}{2m2^{2m}}\sum_{n=1}^m2n\binom{2m}{m-n}\sum_{k=0}^{n-1}s^{k}\Exp[\Tr {V_{2n-2k-1,s}(\tB)}]\notag\\
 & \leqslant \frac{1}{(2m+1)2^{2m}}\sum_{n=1}^m(2n+1)\binom{2m+1}{m-n}(\sum_{k=0}^{n-1}s^{k}\Exp[\Tr {V_{2(n-k),s}(\tB)}]+s^nM)\notag\\
 & \quad +\frac{1}{2m2^{2m}}\sum_{n=0}^m2n\binom{2m}{m-n}\sum_{k=0}^{n-1}s^{k}\Exp[\Tr {V_{2n-2k-1,s}(\tB)}] \label{eq:Am2}
 \end{align}
Invoke \eqref{eq:TraceVks} to get with $C_{M,N}=C_{\hat\Sigma}(1+\sqrt{M/N})$,
\begin{align*}
 A_m & \leqslant \sum_{n=1}^m\frac{2n+1}{(2m+1)2^{2m}}\binom{2m+1}{m-n}
 \\
 &\quad\quad\times\sum_{k=0}^{n-1}s^{k}C_{0,\hat\Sigma}\textcolor{red}{}{{\Big(\frac{MN}{(M-1)(N-1)}\Big)^{n-k}}}2(n-k)\exp\Big[C_{M,N}\frac{2^{\frac32}(n-k)^{\frac32}}{M^{\frac12}}\Big]\\
 & \quad +\sum_{n=1}^m\frac{n}{m2^{2m}}\binom{2m}{m-n}\\
 &\quad\quad\times\sum_{k=0}^{n-1}s^{k}C_{0,\hat\Sigma}{\Big(\frac{MN}{(M-1)(N-1)}\Big)^{n-k-\frac12}}2\Big(n-k-\frac12\Big)\exp\Big[C_{M,N}\frac{2^{\frac32}(n-k-1/2)^{\frac32}}{M^{\frac12}}\Big]\\
 & \quad + \frac{1}{(2m+1)2^{2m}}\sum_{n=0}^m(2n+1)\binom{2m+1}{m-n}s^nM.
 \end{align*}
 Then
 \begin{align*}
A_m & \leqslant 
\frac{2C_{0,\hat\Sigma}}{1-\frac{(M-2)^2}{MN}}\sum_{n=1}^m \Big[\frac{2n+1}{(2m+1)2^{2m}}\binom{2m+1}{m-n}+\frac{n}{m2^{2m-1}}\binom{2m}{m-n}\Big]\\
 &\quad\quad\times n{\Big(\frac{MN}{(M-1)(N-1)}\Big)^n}\exp\Big[C_{M,N}\frac{2^{\frac32}n^{\frac32}}{M^{\frac12}}\Big]\\
  &\quad+ \frac{1}{(2m+1)2^{2m}}\sum_{n=0}^m(2n+1)\binom{2m+1}{m-n}s^nM. 
\end{align*}
\noindent
From Lemma \ref{lem:BorneD1} 
it holds
\[
\log\Big[\frac{n+1/2}{2^{2m}}\binom{2m+1}{m-n}\Big]\vee\log\Big[\frac{n}{2^{2m}}{2m \choose m-n}\Big]\leq -c_1-c_2
\frac{n^2}m
\]
where $c_1=-5$ and $c_2=0.6321$. We deduce that
\begin{align*}
 A_m & \leq\frac{4C_{0,\hat\Sigma}}{1-\frac{(M-1)(N-1)s}{MN}}
 \frac{\exp(-c_1)}m
 \sum_{n=1}^mn{\Big(\frac{MN}{(M-1)(N-1)}\Big)^n}\exp\Big(-c_2\frac{n^2}m+C_{M,N}\frac{2^{3/2}n^{3/2}}{M^{1/2}}\Big)
 \\
 & \quad+ \frac{M\exp(-c_1)}{m}\sum_{n=0}^ms^n\exp\Big(-c_2\frac{n^2}m\Big)\,,
 \\
 & \leq\frac{4C_{0,\hat\Sigma}\exp(-c_1)}{1-\frac{M}{N}}
 \Big[{\Big(\frac{MN}{(M-1)(N-1)}\Big)^m}+\frac Mm\Big]
 \sum_{n=1}^m\exp\Big(-c_2\frac{n^2}m+C_{M,N}\frac{2^{3/2}n^{3/2}}{M^{1/2}}\Big)\,.
\end{align*}
Observe that the maximum of $-ax^4+bx^3$ is $\frac{27b^4}{256a^3}$. We deduce that 
\[
-c_2\frac{n^2}m+C_{M,N}\frac{2^{3/2}n^{3/2}}{M^{1/2}}\leq C_{\mathrm{Rad}}(1+\sqrt{M/N})^4\frac{m^3}{M^2}
\]
where 
\[
C_{\mathrm{Rad}}=\frac{27}{4}\frac{C_{M,N}^4}{c_2^3(1+\sqrt{M/N})^4}=\frac{27}{4}\frac{C_{\hat\Sigma}^4}{c_2^3}=\frac{27}{4}\frac{1.81^4C_{D}^2}{c_2^3}=286.9C_{D}^2\,,
\]
as claimed.

The bound on $B_m:=\Exp[\Tr{\tB^{2m}}]-\Exp[\Tr{\tB^{2m-1}}]$ follows the same lines. The minus in front of $\Exp[\Tr{\tB^{2m-1}}]$ change the line \eqref{eq:Am1} to its opposite. The change of indices $k$ leads to the term $s^{k+1/2} \Exp[\Tr{V_{2(n-k-1),s}(\tB)}]$ in \eqref{eq:Am2}. Since we uniformly bound $n-k-1$ by $n$ in the rest of the proof and $s^{1/2}<1$, we get the same result.
\end{proof}

\subsection{Small deviation on the largest eigenvalue}
\label{sec:SmallDevLargest}
Observe that
\begin{align*}
 \Prob\{\lambda_M(\B)\geqslant (\sqrt{M}+\sqrt{N})^2+\varepsilon N\} & = \Prob\{\lambda_M(\tB)\geqslant \varepsilon_{M,N}\},
\end{align*}
with 
\[
\frac{55}{53}\Big(1+\frac{\varepsilon}{2\sqrt{M/N}}\Big)\geq
\varepsilon_{M,N}:=\frac{\sqrt{MN}+1}{\sqrt{(M-1)(N-1)}}+\frac{{\varepsilon}N}{2\sqrt{(M-1)(N-1)}}\geq1+\frac{\varepsilon}{2\sqrt{M/N}}\,,
\]
for all $N>M\geq54$. Set $f(x):=x^{2m}+x^{2m-1}$ and note that $f$ is non-increasing on $(-\infty,-1+\frac{1}{2m}]$ and non-decreasing on $[-1+\frac{1}{2m},\infty)$. Furthermore, its minimum is $-e_m$ where 
\[
e_m:=\frac{(2m-1)^{2m-1}}{(2m)^{2m}}=\frac{(1-\frac1{2m})^{2m}}{2m-1}\leq \frac1{2em}\,,
\]
and it is non-negative on $(-\infty,-1]\cup[0,\infty)$. Using Markov inequality, we deduce that
\begin{align}
\Prob(\lambda_M(\tB)\geqslant \varepsilon_{M,N})
 & \leq \Prob(f(\lambda_M(\tB))+e_m\geqslant f(\varepsilon_{M,N})+e_m)\notag\\
 & \leqslant \frac{\Exp[f(\lambda_M(\tB)]+e_m}{f(\varepsilon_{M,N})+e_m}\notag\\
 & \leqslant \frac{\sum_{k=1}^M(\Exp[f(\lambda_k(\tB)]+e_m)}{f(\varepsilon_{M,N})}\notag\\
 & = \frac{A_m+Me_m}{f(\varepsilon_{M,N})}
 \label{eq:MarkovLargest}
 \end{align}
Invoke Lemma \ref{lem:traces} to get that
 \begin{align*}
\Prob(\lambda_M(\tB)\geqslant \varepsilon_{M,N})
 & \leq \frac{{C_{0,\mathrm{Rad}}}m
 \big[{\Big(\frac{MN}{(M-1)(N-1)}\Big)^m}+\frac Mm\Big]
\exp\Big( C_{\mathrm{Rad}}(1+\sqrt{M/N})^4\frac{m^3}{M^2}\Big)+\frac{M}{2em}}{({1-\frac{M}{N}})f(\varepsilon_{M,N})}\,,
\end{align*}
for all $m \in \N$. Using that $M\geq54$ and $\log(1+x)\leq x$, we get
\begin{align*}
\Prob(\lambda_M(\tB)\geqslant \varepsilon_{M,n})
 & \leq \frac{{C_{0,\mathrm{Rad}}}
 \big[m+\frac{1+2e}{2e}M\big]
 \,
e^{ C_{\mathrm{Rad}}(1+\sqrt{M/N})^4\frac{m^3}{M^2}+54m(\frac1M+\frac1N)\log(\frac{54}{53})}}{({1-\frac{M}{N}})f(\varepsilon_{M,N})}
\end{align*}
for all $m \in \N$. Optimizing on $m$ yields the choice $m=\sqrt{\frac{2\log(\varepsilon_{M,N})}{3C_{\mathrm{Rad}}(1+\sqrt{M/N})^4}}M$ and
\[
\Prob\Big\{\lambda_M(\B)  \geqslant (\sqrt{M}+\sqrt{N})^2+\varepsilon N\Big\}
\leq
\frac{\mathds W_0(\rho,\varepsilon)}
{{1-\rho}}
M
\exp(-N\mathds W_1(\rho,\varepsilon))
\]
where $\rho=M/N$ and 
\begin{align*}
\mathds W_0(\rho,\varepsilon)&:=
\frac
{C_{0,\mathrm{Rad}}{(1+2e){\sqrt{3C_{\mathrm{Rad}}}}}(1+\sqrt{\rho})^2+{2e{\sqrt{2\log(\frac{55}{53}(1+\frac{\varepsilon}{2\sqrt{\rho}}))}}}}
{2e\sqrt{3C_{\mathrm{Rad}}}(1+\sqrt{\rho})^2}
\\
&\quad\times \exp
\left[
54\log(\frac{54}{53})(1+\rho){\frac{\sqrt{2\log(\frac{55}{53}(1+\frac{\varepsilon}{2\sqrt{\rho}}))}}{(1+\sqrt{\rho})^2\sqrt{3C_{\mathrm{Rad}}}}}
\right]
\\
\mathds W_1(\rho,\varepsilon)&:=
\frac{4\sqrt{2}}{3\sqrt{3}}\frac{\rho\log(1+\frac{\varepsilon}{2\sqrt{\rho}})^{\frac32}}{(1+\sqrt{\rho})^2\sqrt{C_{\mathrm{Rad}}}}
\end{align*}
Using that $\rho\leq1$, we derive that 
\begin{align*}
\mathds W_0(\rho,\varepsilon)
&\leq\frac
{4C_{0,\mathrm{Rad}}{(1+2e){\sqrt{3C_{\mathrm{Rad}}}}}+{2e{\sqrt{2\log(\frac{55}{53}(1+\frac{\varepsilon}{2\sqrt{\rho}}))}}}}
{2e\sqrt{3C_{\mathrm{Rad}}}}
\\
&\quad\times
\exp
\left[
108\log(\frac{54}{53}){\frac{\sqrt{2\log(\frac{55}{53}(1+\frac{\varepsilon}{2\sqrt{\rho}}))}}{\sqrt{3C_{\mathrm{Rad}}}}}
\right]
\\
&\leq c_0\exp
\left[
{c_0\sqrt{\log\Big(1+\frac{\varepsilon}{2\sqrt{\rho}}\Big)}}
\right]
\end{align*}
for some universal constant $c_0>0$. We deduce the following useful bound
\eq
\label{eq:SmalDevLogLarge}
\Prob\Big\{\lambda_M(\B)  \geqslant (\sqrt{M}+\sqrt{N})^2+\varepsilon N\Big\}
\leq
\frac{c_0M
e^{
{c_0\sqrt{\log\big(1+\frac{\varepsilon}{2\sqrt{\rho}}\big)}}
}}
{{1-\rho}}
e^{-N\mathds W_1(\rho,\varepsilon)}\,.
\qe
For $\varepsilon\leq \sqrt{\rho}$ we can deduce a small deviation inequality as follows. Observe that for any $\eta>0$ one can pick a constant $c_1(\eta)>0$, that depends only on $\eta$, such that for all $M\geq1$, it holds $M\leq c_1(\eta)\exp(\eta M)$. Note that
$\log(3/2)\frac{\varepsilon}{\sqrt{\rho}}
\leq
\log\big(1+\frac{\varepsilon}{2\sqrt{\rho}}\big)$
and set
\[
\mathbb V_{\mathrm{Rad}}:=
\frac{3\sqrt{3C_{\mathrm{Rad}}}}
{4\sqrt{2}\log(3/2)^{3/2}}
\,.
\]
We deduce that for any $C>\mathbb V_{\mathrm{Rad}}\approx 3242$ there exists a constant $v:=v(\rho,C)>0$ that depends only on $\rho=M/N$ and $C$ such that, for all $0\leq\varepsilon\leq \sqrt{\rho}$,
\eq
\label{eq:SmalDevLarge}
\Prob\Big\{\lambda_M(\B)  \geqslant (\sqrt{M}+\sqrt{N})^2+\varepsilon N\Big\}
\leq
v
\exp\Big(-C^{-1}N\frac{\rho^{1/4}}{(1+\sqrt{\rho})^2}\varepsilon^{\frac32}\Big)\,.
\qe

\subsection{Small deviation on the smallest eigenvalue}
Observe that
\begin{align*}
 \Prob\{\lambda_1(\B)\leqslant (\sqrt{M}-\sqrt{N})^2-\varepsilon' N\} & = \Prob\{\lambda_1(\tB)\leqslant -\varepsilon'_{M,N}\},
\end{align*}
with 
\[
1+\frac{\varepsilon'27}{53\sqrt{M/N}}\geq
\varepsilon'_{M,N}:=\frac{\sqrt{MN}-1}{\sqrt{(M-1)(N-1)}}+\frac{{\varepsilon'}N}{2\sqrt{(M-1)(N-1)}}\geq\frac{53}{54}\Big(1+\frac{\varepsilon'27}{53\sqrt{M/N}}\Big)\,,
\]
for all $N>M\geq54$. Set $g(x):=x^{2m}-x^{2m-1}$ and note that $g(x)=f(-x)$. It holds
\begin{align}
\Prob\{\lambda_1(\tB)\leqslant -\varepsilon'_{M,N}\}
 & \leq \Prob\{g(\lambda_1(\tB))+e_m\geqslant g(-\varepsilon'_{M,N})+e_m\}\notag\\
 & \leqslant \frac{\Exp[g(\lambda_1(\tB)]+e_m}{f(\varepsilon'_{M,N})+e_m}\notag\\
 & \leqslant \frac{\sum_{k=1}^M(\Exp[g(\lambda_k(\tB)]+e_m)}{f(\varepsilon'_{M,N})}\notag\\
 & = \frac{B_m+Me_m}{f(\varepsilon'_{M,N})}\notag
 \end{align}
and we recover an upper bound of the form \eqref{eq:MarkovLargest} for which Lemma \ref{lem:traces} can also be applied and we get that
 \begin{align*}
\Prob\{\lambda_1(\tB)\leqslant -\varepsilon'_{M,N}\}
 & \leq \frac{{C_{0,\mathrm{Rad}}}m
 \big[{\Big(\frac{MN}{(M-1)(N-1)}\Big)^m}+\frac Mm\Big]
\exp\Big( C_{\mathrm{Rad}}(1+\sqrt{M/N})^4\frac{m^3}{M^2}\Big)+\frac{M}{2em}}{({1-\frac{M}{N}})f(\varepsilon'_{M,N})}\,,
\end{align*}
for all $m \in \N$. The rest of the proof follows the same lines as in Section \ref{sec:SmallDevLargest} where we change $\varepsilon_{M,N}$ by $\varepsilon'_{M,N}$, we choose $m=\sqrt{\frac{2\log(54\varepsilon'_{M,N}/53)}{3C_{\mathrm{Rad}}(1+\sqrt{M/N})^4}}M$ and may have changed the harmless constant $c_0$ in $\mathds W_0$. Eventually, note that \eqref{eq:SmalDevLarge} has been obtained from \eqref{eq:SmalDevLogLarge} and we can use the same argument for the deviation on the smallest eigenvalue. This proves Proposition \ref{prop:FS}. 

\section{Stirling's formula and bounds on binomial coefficients}
\begin{lem}
\label{lem:AsymptoticGamma}
Let $z>0$ then there exists $\theta\in(0,1)$ such that:
\[
\Gamma(z+1)=(2\pi z)^{\frac12}\big(\frac ze\big)^{z}\exp\big(\frac\theta{12z}\big)\,.
\]
\end{lem}
\begin{proof}
See \cite{abramowitz1965stegun} Eq. 6.1.38.
\end{proof}

\begin{lem}
\label{lem:BorneD1}
It holds, for all $1\leqslant n \leqslant m$,
\begin{align*}
\log\Big[\frac{n}{2^{2m}}{2m \choose m-n}\Big]&\leq5-0.6321
\frac{n^2}m\\
\log\Big[\frac{n+1/2}{2^{2m}}\binom{2m+1}{m-n}\Big]&\leq2-0.6555
\frac{n^2}m
\end{align*}
\end{lem}
\begin{proof}
If $n=m$ then the result is clear. Otherwise, using Lemma \ref{lem:AsymptoticGamma}, one has
\begin{align*}
\log\Big[\frac{n}{2^{2m}}{2m \choose m-n}\Big]&\leqslant -0.364+\log n+(2m+1/2)\log m
\\
&\quad-(m-n+1/2)\log(m-n)-(m+n+1/2)\log(m+n)\,,
\\
&\leqslant -0.364-1/2\log( (m^{2}-n^{2})/(mn^{2}))
\\
&\quad+m\Big[\frac nm\log(1-\frac {2n/m}{1+n/m})-\log(1-\big(\frac nm\big)^{2})\Big]\,.
\end{align*}
The last term in the right hand side can be upper bounded thanks to the identity $x\log(1 - 2 x/(1 + x))- \log(1 - x^2)\leq-x^{2}$ for all $0<x<1$. It yields
\[
m\Big[\frac nm\log(1-\frac {2n/m}{1+n/m})-\log(1-\big(\frac nm\big)^{2})\Big]\leq-\frac{n^2}m\,.
\]
Let $x=n/m$ and observe that $x\leq1-1/m$. It holds that the middle term of the aforementioned right hand side can be expressed as
\[
-1/2\log( (m^{2}-n^{2})/(mn^{2}))=1/2\log(mx^2/(1-x^2))\,.
\]
If $x\leqslant0.99995$ then, using that $\log(z)\leqslant z/e$, we have
\[
1/2\log(mx^2/(1-x^2))\leqslant4.6052+(1/(2e))mx^2\,.
\]
If $0.99995<x\leq1-1/m$ then 
\[
1/2\log(mx^2/(1-x^2))\leqslant\log m\leqslant m/e<0.3679mx^2\,.
\]
In all cases, we get that
\[
1/2\log(mx^2/(1-x^2))\leqslant4.6052+0.3679mx^2
\]
We deduce that
\[
\log\Big[\frac{n}{2^{2m}}{2m \choose m-n}\Big]\leqslant 4.24
-0.6321
n^2/m \,,
\]
as claimed.
\end{proof}
\end{document}